\documentclass[reqno,11pt]{amsart}
\usepackage{graphicx} 
\usepackage{hyperref}
\usepackage{amsmath,amsfonts,amssymb,amsthm,epsfig,amscd}
\usepackage{color,comment}
\usepackage[]{fontenc}
\usepackage[latin1]{inputenc}
\usepackage{enumerate}
\usepackage[cmtip,all]{xy}
\usepackage{tabularx,multirow}
\usepackage{accents}
\usepackage{mathtools}
\usepackage{fullpage}
\allowdisplaybreaks 
\numberwithin{equation}{section}

\newtheorem{theorem}{Theorem}[section]
\newtheorem{proposition}[theorem]{Proposition}
\newtheorem{corollary}[theorem]{Corollary}
\newtheorem{lemma}[theorem]{Lemma}

\newtheorem{claim}[theorem]{Claim}

\newtheorem{theoremalpha}{Theorem}
\newtheorem{corollaryalpha}[theoremalpha]{Corollary}

\theoremstyle{definition}

\newtheorem{remark}[theorem]{Remark}

\theoremstyle{remark}

\DeclareMathOperator{\codim}{codim}

\DeclareMathOperator{\Hom}{Hom}

\DeclareMathOperator{\rk}{rk}

\def\GL{\mathrm{GL}}
\def\bP{{\mathbb P}}
\def\bG{{\mathbb G}}

\def\d{{\mathbf d}}
\def\F{{\mathbf F}}
\def\s{{\mathbf s}}

\title{Subvarieties of complete intersections of large degree}
\author{Francesco Bastianelli} 
\address{ Dipartimento di Matematica, Universit\`a degli Studi di Bari Aldo Moro,
Via Edoardo Orabona 4, 70125 Bari, Italy}
\email{francesco.bastianelli@uniba.it}

\author{Gianluca Pacienza}
\address{
Universit\'e de Lorraine, CNRS, IECL, 
F-54000 Nancy, France}
\email{gianluca.pacienza@univ-lorraine.fr}
\thanks{G.P. is partially supported by ANR project No. ANR-23-CE40-0026 (POK0). F.B. is a member of INdAM (GNSAGA)}

\begin{document}

\begin{abstract}
We study subvarieties of very general complete intersections $X\subset \bP^n$ of multidegree $(d_1,\dots,d_c)$, when $d\coloneqq d_1+\dots +d_c$ is sufficiently large.
In a seminal paper Ein proved that if $d\geq 2n-c-k+2$, any $k$-dimensional subvariety of $X$ is of general type and has positive geometric genus. We strengthen this result by obtaining the optimal bound $d\geq 2n-c-k$, provided that $n> 2c+k$. 
As a consequence, we characterize algebraic hyperbolicity of very general complete intersections $X\subset \bP^n$ of codimension $c\leq \frac{n-3}{2}$. 
For lower values of $d$, we prove that if $\frac{3n-c+2}{2}\leq d\leq 2n-c-2$ and $(d_1,\dots,d_c)$ satisfies an additional numerical condition, then the only curves in $X$ that are not of general type are lines. Moreover, we describe the locus where positive dimensional orbits of points under rational equivalence must lie.
We obtain our results by proving that, under suitable numerical conditions, subvarieties of $X$ that are not of general type must lie in the locus of $X$ covered by lines.
The proof of this result relies on a generalization of the approach and techniques developed for hypersurfaces by Voisin, Clemens-Ran and the second author, combined with a Grassmannian technique introduced by Riedl-Yang.  
\end{abstract}

\maketitle

\section{Introduction}
In this paper, we study subvarieties of very general complete intersections of the complex projective space.
The description of ``special'' subvarieties is a fundamental issue in the classification of algebraic varieties, which is strictly related to important problems, as for instance algebraic hyperbolicity or the study of Chow groups.
Furthermore, complete intersections in projective spaces provide some of the most natural examples of varieties of arbitrary dimension and codimension in a fixed ambient space. Since, by definition, they come with a full set of equations, complete intersections have been used to test conjectures, develop techniques, prove neat results which may be out of reach in general. 
As a bright example of this we may quote a seminal paper \cite{E88} by Ein in which he initiated the study of the geometry of subvarieties  of very general complete intersections of the complex projective space. 

To state Ein's result let us fix some notation. Let $X\subset \mathbb{P}^n$ be a very general complete intersection of type $\mathbf{d}=(d_1,\dots, d_c)$, with 
$$
d\coloneqq d_1+\dots+d_c\geq n+2 
\quad \text{and} \quad 
2\leq d_1\leq \dots \leq d_c. 
$$ 
Ein proved that if $d\geq 2n-c-k+2$, then all the $k$-dimensional subvarieties of $X$ are of general type, and when $d= 2n-c-k+1$, any subvariey of $X$ admits a desingularization with positive geometric genus.
On the other hand, if $d\leq 2n-c-k-1$, the union of the lines contained in $X$ describes a $k$-dimensional subvariety of $X$, which is of course not of general type and has geometric genus zero (cf. \cite{DM}). 
Therefore it remains to describe $k$-dimensional subvarieties of $X$ when $d= 2n-c-k$ and $d= 2n-c-k+1$.

In the case of hypersurfaces, the problem is well understood. 
Indeed, it follows from the fundamental works by Voisin \cite{V1,V2}, and from  \cite{C03, CR04,P03,P04}, that if $X$ is a very general hypersurface of degree $d= 2n-c-k  $ (with $k\leq n-5$) or $d= 2n-c-k+1$ (with $k\leq n-3$), all the $k$-dimensional subvarieties of $X$ are of general type and with positive geometric genus. Despite an intense activity and several new developments and contributions in more recent years in the case of hypersurfaces, 
see e.g. \cite{Abe23, CR19,CR23,HI21, KS25, Mio25, MY24, RY20, RY22, Yeo25}, much less progress was made for complete intersections of arbitrary codimension, for a recent contribution see \cite{DR25}. The goal of the present paper is to fill this gap, by extending some of the main techniques and results available only for hypersurfaces to this more general context, with the hope that this will be useful beyond the applications presented here. 

First of all we generalize the results included in \cite{V1,V2, C03, CR04,P03,P04} to very general complete intersections of larger codimension, hence sharpening Ein's theorem. 

\begin{theoremalpha}\label{thm:Ein+}
Let $a\in \{0,1\}$ and $n,c,k$ be positive integers such that 
$n\geq \max\{2c+k+a, c+k+3+2a\}$. Let $X\subset \mathbb{P}^n$ be a very general complete intersection of multidegree $(d_1,\dots,d_c)$ with $d\geq 2n-c-k$.
Let $Y\subset X$ be any $k$-dimensional subvariety, and  $\nu\colon\widetilde{Y}\longrightarrow Y$ be a desingularization. Then $h^0 \left(\widetilde{Y}, K_{\widetilde{Y}}\otimes \nu^*\mathcal O_{{Y}}(-a)\right)>0$.
\end{theoremalpha}
In particular, under the numerical assumptions of Theorem \ref{thm:Ein+}, 
when $a=0$ (respectively $a=1$), any $k$-dimensional subvariety of $X$ is not covered by rational curves (resp. is of general type).

\smallskip
This type of results is closely related to hyperbolicity properties of algebraic varieties.
We recall that a projective variety $M$ is \emph{algebraically hyperbolic} \`a la Demailly (respectively \`a la Lang) if there exist a real number $\varepsilon >0$ and an ample line bundle $H$ such that any integral curve $C\subset M$ satisfies $2g(C)-2\geq \varepsilon (H\cdot C)$ (respectively if any subvariety $Y$ of $M$ is of general type). 
For a projective variety, algebraic hyperbolicity \`a la Demailly is implied by Kobayashi hyperbolicity (see \cite{Dem97}), and the converse is conjectured to be true, see \cite[Section 3]{Dem20} (this is known for subvarieties of abelian varieties, see e.g. \cite[Remark 2.2]{Cau25}). 
Also algebraic hyperbolicity \`a la Lang is conjectured to be equivalent to hyperbolicity  and its sufficiency is implied by the Green-Griffiths conjecture, see \cite[Conjectures 5.5 and 5.6]{L86}.

As recalled above, a very general complete intersection $X\subset \mathbb{P}^n$ with $d\leq 2n-c-2$ is not algebraically hyperbolic, in neither sense, as $X$ contains lines.
As an immediate consequence of Theorem \ref{thm:Ein+} and Ein's result, we characterize the algebraic hyperbolicity for very general complete intersections, under some hypothesis on the codimension $c$. 

\begin{corollaryalpha}\label{cor:DR+} 
Let $n,c$ be positive integers, and let $X\subset \mathbb{P}^n$ be a very general complete intersection of multidegree $(d_1,\dots, d_c)$.
Then $X$ is algebraically hyperbolic \`a la Demailly (respectively \`a la Lang) if and only if $d\geq 2n-c-1$, provided that $n\geq \max\{2c+2, c+5\}$ (resp.  $n\geq \max\{2c+3, c+7\}$).
\end{corollaryalpha}
\smallskip
The statement was proved for $c=1$ in \cite[Appendix A]{CR04} and \cite[Corollary 1.2]{P04}.  
Furthermore, in the recent paper \cite[Corollary 1.3]{DR25}, the authors proved that if $X\subset \mathbb{P}^n$ is a very general complete intersection such that $d\geq 2n-c$, then $X$ is algebraically hyperbolic \`a la Demailly. 
We also point out that Corollary \ref{cor:DR+} gives a confirmation (and a sharpening by one) of \cite[Conjecture 0.18]{Dem20}.

\smallskip
Turning to lower degrees, as any complete intersection contains lines when $d\leq 2n-c-2$,  it is natural to wonder whether a very general one may contain rational curves other than lines.
In the case $c=1$, Riedl and Yang \cite{RY20} proved that lines are the only rational curves contained in a very general hypersurface $X\subset \mathbb{P}^n$  of degree $(3n+1)/2\leq d\leq 2n-3$, and this result has been sharpened to $d\geq 3n/2$ by Coskun and Riedl in \cite{CoR22}. Using the main theorem of \cite{BCFS20}, we generalize  the  result by Riedl and Yang to complete intersections of arbitrary codimension, 
and we also discuss the existence of elliptic curves.

\begin{theoremalpha}\label{thm:RY+} 
Let $a\in \{0,1\}$ and $n,c$ be positive integers, and let $X\subset \mathbb{P}^n$ be a very general complete intersection of type $(d_1,\dots,d_c)$, with
$$n+1+a+\left\lfloor\frac{n-c}{2}\right\rfloor\leq d\leq 2n-c-2\qquad \text{and}\qquad 
\sum_{i=1}^c \frac{d_i(d_i+1)}{2}-a(d_1+d_c)\geq 3n-2.  
$$
If $a=0$, then the only rational curves contained in $X$ are lines, and if $a=1$ then $X$ does not contain elliptic curves.
\end{theoremalpha}
Voisin conjectured  \cite[Conjecture 3.9]{V03} that the degrees
of rational curves on a very general hypersurface of general type are bounded. Theorem \ref{thm:RY+} provides some evidence that the same could be expected for complete intersections.
We note that if $a=0$ and $c=1$, the quadratic condition on the $d_i$'s in the statement of Theorem \ref{thm:RY+} is implied by the linear lower bound on $d$, so that one recovers the result by Riedl and Yang \cite{RY20}. 
More generally, the same implication holds whenever $n\geq 3c$ (cf. Remark \ref{rem:(1.2-2)}). 

\smallskip
The study of rational curves can of course be seen as a particular case of the study of the orbits of points under rational equivalence. If $X\subset \mathbb P^n$ is a very general complete of type $(d_1,\dots,d_c)$ with $2n-c-d\geq  0$,
Chen, Lewis and Sheng recently conjectured in particular that the locus of points rationally equivalent to a fixed one  has dimension 
at most $2n-c-d$ (see \cite[Conjecture 1.4]{CLS21} for the precise and more general statement, and for a discussion of results and motivations). The conjecture was proved by Riedl and Yang \cite[Theorem 1.6]{RY22}. The techniques introduced by Voisin in \cite{V1,V2} in the case $c=1$ that we extend to any $c\geq 1$ are powerful enough to allow to study not only subvarieties with vanishing geometric genus but more generally subvarieties whose points are rationally equivalent in the ambient complete intersection. As a by-product we obtain a result that can be seen as a strengthening of the conjecture by Chen-Lewis-Sheng in the sense that, under some numerical conditions, not only it bounds the dimension of orbits of points under rational equivalence, but it tells exactly where positive dimensional orbits must lie. To state the result, 
for an integer $r\geq 1$, a $c$-tuple $\F=(F_1,\ldots, F_c)$ of homogeneous polynomials of degrees $(d_1,\dots,d_c)$, and $X_\F$ the complete intersection defined by $\F$, we define the locus 
\begin{equation}\label{eq:deltraintro}
\Delta_{r,\bf F}:=\left\{x\in X_{\bf F}\left|\begin{array}{l}  \textrm{there exist a line }\ell \textrm{ passing through } x \textrm{ such that } \\\ell\cdot V(F_i)\geq rx \text{ if }d_i=r, \text{ and } \ell\subset V(F_i)\text{ otherwise}\end{array}\right.\right\}.
\end{equation}
For very general $\F$ we have
$$\dim \Delta_{r,\bf F} = 2n-1 - \sum_{d_j\not=r} (d_j+1) - \sum_{d_i=r} d_i,$$ so, in particular, this locus has dimension at most $2n-c-d$ if the degrees are all distinct.

\begin{theoremalpha}\label{thm:CLS+}
  Let $n,c,k$ be positive integers,  and let $X\subset \mathbb{P}^n$ be a very general complete intersection  of type $(d_1,\dots,d_c)$, with 
  $$
  n+1+\left\lfloor\frac{n-c-k+1}{2}\right\rfloor\leq d\leq 2n-c. 
  $$ 
  Then any $k$-dimensional subvariety $Y\subset X$ whose points are all rationally equivalent in $X$ lies in
  $\Delta_{r,\bf F}$, for some integer $r\geq 1$. 
\end{theoremalpha}
For $c=1$ the picture is clear: 
the locus $\Delta_{d,\F}$ has dimension precisely $2n-1-d$ and, since in this case its points are all rationally equivalent, it is the only positive dimensional orbit under rational equivalence.
When $c\geq 2$ our result leads to the following.
\smallskip

\noindent
{\bf Question.} When $
  2n-c> d$, is it possible to give a  geometric description of the rational equivalence orbits of the largest dimension on a very general complete intersection of $\mathbb{P}^n$ of type $(d_1,\dots,d_c)$?

\smallskip
All the above results rely on the following theorem or are obtained along its proof.
\begin{theoremalpha}\label{thm:main}
Let $n,c,k$ be positive integers, let $a\in\{0,1\}$, and let $X\subset \mathbb{P}^n$ be a very general complete intersection of type $(d_1,\dots,d_c)$, with 
\begin{equation}\label{eq:bound lineare}
d\geq n+1+\left\lfloor\frac{n-c-k+1}{2}\right\rfloor   + a
\end{equation}
and 
\begin{equation}\label{eq:bound quadratico}
\sum_{i=1}^c \frac{d_i(d_i+1)}{2}-a(d_{1}+d_c)\geq 3n-k-1.  
\end{equation}
If $Y\subset X$ is a $k$-dimensional subvariety such that 
\begin{equation}\label{eq:van}
h^0 \left(\widetilde{Y}, K_{\widetilde{Y}}\otimes \nu^*\mathcal O_{{Y}}(-a)\right)=0,
\end{equation}
where $\nu\colon\widetilde{Y}\longrightarrow Y$ is a desingularization, then $Y$ is contained in the union of the lines lying on $X$.
\end{theoremalpha}

We point out that condition (\ref{eq:bound quadratico}) can be slightly weakened (cf. Remarks \ref{rem:(1.2-0)} and \ref{rem:(1.2-1)}), and it can  be further improved in a more substantial way when $a=0$ and all the $d_i$'s are distinct (see Theorem \ref{thm:main-distinc}). 
We also note that Theorem \ref{thm:main} extends to complete intersections analogous results in the case of hypersurfaces (cf. \cite[Theorem 1.1]{C03}, \cite[Theorems 0.1 and A.1]{CR04} and \cite[Theorem 1.1]{P04}). Moreover, the case $c=1$ and $a=0$ has been recently improved in \cite{Abe23}.

\smallskip
Voisin's approach in the case of hypersurfaces \cite{V1,V2} was greatly extented by Clemens \cite{C03}, leading to the results by Clemens-Ran \cite{CR04} and the second author \cite{P04}, and it has been slightly refined more recently in \cite{Abe23}. 
Our main contribution here is, on the one hand, to generalize this approach to complete intersections, overcoming the technical difficulties that arise along the way.
On the other hand, this approach breaks down for $c>1$ when $a=1$, or when $a=0$ and some of the $d_i$'s coincide, see Remark \ref{rmk:nonfunziona col -1}.
To circumvent this obstacle we adapt to our setting the Grassmannian technique introduced by Riedl and Yang \cite{RY20} (see also \cite{RY22, CoR22}).

To give more details, we consider the universal family $\mathcal{X}\longrightarrow U$ of complete intersections, where $U$ is the open set parameterizing $c$-tuples $\F=(F_1,\dots,F_c)$ of homogeneous polynomials of degrees $(d_1,\dots,d_c)$, such that $X_{\F}\subset \bP^n$ is a 
complete intersection.
Up to base-changing $\mathcal{X}$ by a suitable \'etale cover of $U$, we may consider a subfamily $\mathcal{Y}\subset \mathcal{X}\subset \bP^n\times U$ such that the fibre over a general $\F\in U$ corresponds to a $k$-dimensional subvariety $Y_{\F}\subset X_{\F}\subset \bP^n$ which satisfies \eqref{eq:van}.
Denoting by $\pi\colon{\mathcal Y}\longrightarrow {\mathbb P^n}$ the first projection, we define the vertical part of the tangent sheaf $T_{\mathcal Y}$ by the exact sequence
$$
0\longrightarrow T^\textrm{vert}_{\mathcal Y}\longrightarrow  T_{\mathcal Y}\ {\buildrel {\pi_*}\over \longrightarrow} T_{\mathbb P^n}\longrightarrow 0.
$$
By generalizing to complete intersections the multiplication maps introduced in \cite{C03} and using \eqref{eq:bound lineare}, we prove that for general $(y,\F)\in \mathcal{Y}$, there exists a unique line $\ell_{(y,\F)}\subset \bP^n$ passing through $y\in Y_{\F}$ such that $T^\textrm{vert}_{\mathcal Y,(y,\F)}\subset U$ has large intersection with the space $\bigoplus_{i=1}^c H^0\big(\mathbb{P}^n, \mathcal{I}_{\ell_{(y,\mathbf{F})}}(d_i)\big)$ of $c$-tuples vanishing on $\ell_{(y,\F)}$.

As we vary $(y,\F)\in \mathcal{Y}$, the spaces  $T^\textrm{vert}_{\mathcal Y,(y,\F)}\cap \bigoplus_{i=1}^c H^0\big(\mathbb{P}^n, \mathcal{I}_{\ell_{(y,\mathbf{F})}}(d_i)\big)$ give a distribution on $\mathcal{Y}$, that we prove to be integrable (see Proposition \ref{prop:b_r}).
Using this fact and the description of the leaves of the corresponding foliation, we prove that if $a=0$, there exists an integer $r\geq 1$ such that the line $\ell_{(y,\F)}$ is contained in any hypersurface $V(F_i)$ with $d_i\neq r$, and $\ell_{(y,\F)}$ intersects $V(F_i)$ only at $y$ if $d_i=r$. 
Therefore, we deduce that $Y_{\F}$ is contained in the locus $\Delta_{r,\bf F}$ defined in \eqref{eq:deltraintro}.
Similarly, When $a=1$, we prove that  $Y_{\F}$ is contained in the locus
$$
\Lambda_{r,\F}:=\left\{x\in X_\F\left|\begin{array}{l}
     \textrm{there exist a line }\ell \textrm{ through } x \textrm{ and a point }x'\in \ell \textrm{ such that} \\ \ell \subset V(F_i)\textrm{ if } d_i<r, \textrm{ and } \ell \cdot V(F_i) \geq r x +(d_i-r) x'\textrm{ if } d_i \geq r    \end{array} \right.\right\}
$$
for some integer $r\geq 1$ (see Proposition \ref{prop:Delta_r}).

Then we focus on certain varieties dominating $\Delta_{r,\bf F}$ and $\Lambda_{r,\F}$, that we construct as zero loci of sections of globally generated vector bundles. 
In particular, we describe their canonical bundle, whose positivity is granted by \eqref{eq:bound quadratico}.
So we apply the Grassmannian technique introduced in \cite{RY20} to these loci, and we conclude that $\Delta_{r,\bf F}$ and $\Lambda_{r,\F}$ cannot contain subvarieties $Y_{\F}$ as above lying outside the union of the lines in $X_{\F}$, which is the assertion of Theorem \ref{thm:main}. 

\smallskip
The paper is organized as follows.
In Section \ref{s: preliminaries} we introduce our notation and we prove some preliminary results.
Section \ref{s:contact} is devoted to the study of various loci dominating $\Delta_{r,\bf F}$ and $\Lambda_{r,\F}$.
In Section \ref{sec:CVmaps} we study the multiplication maps through which a unique line (and an unexpectedly large part of its homogeneous ideal)   will appear at each point $(y,\F)$ of a subfamily $\mathcal Y$ satisfying \eqref{eq:van}. In Section \ref{s:int} we prove the integrability of the associated distribution, establishing along the way Theorem \ref{thm:CLS+}. In Section \ref{s:proof Thm E}  we put the Grassmannian technique to work in our situation and prove our main result Theorem \ref{thm:main}. In the last section we show how to deduce the other results from Theorem \ref{thm:main}.

\section{Setting and preliminaries}\label{s: preliminaries}

\subsection{Notation}\label{subs:not}
We work throughout over the field $\mathbb{C}$ of complex numbers.
We say that a property holds for a \emph{general} (resp. \emph{very general}) point ${x\in X}$ if it holds on a Zariski open nonempty subset of $X$ (resp. on the complement of the countable union of proper subvarieties of $X$).

Given an irreducible variety $Y$ endowed with a morphism $f\colon Y\longrightarrow \mathbb{P}^n$ and a coherent sheaf $\mathcal{F}$ on $Y$, for any $i\in \mathbb{Z}$ we set 
$$
\mathcal{F}(i):=\mathcal{F} \otimes f^{*}\mathcal{O}_{\mathbb{P}^n}(i).
$$

Consider two integers $n\geq c\geq 1$ and a $c$-tuple $\mathbf{d}=(d_1,\dots, d_c)$ of integers $1\leq d_1\leq \cdots \leq d_c$, with
$$
d:=|\mathbf{d}|=d_1+\dots+d_c.
$$ 
For any $i=1,\dots,c$, we set
$$
S^{d_i}:=H^0\left(\mathbb{P}^n,\mathcal{O}_{\mathbb{P}^n}(d_i)\right), \quad
S^{\mathbf{d}}:= \bigoplus_{i=1}^c S^{d_i}\quad \textrm{and}\quad N:= \dim{S^{\mathbf{d}}}.
$$
Hence $\F=(F_1,\ldots, F_c)\in S^\d$ will denote a $c$-tuple of homogeneous polynomials of multidegree $(d_1,\ldots, d_c)$. 
Given a subvariety $Z\subset \mathbb{P}^n$, we also consider the subspaces of sections vanishing on $Z$,
\begin{equation}\label{eq:S_Z}
S_Z^{d_i}:=H^0\left(\mathbb{P}^n,\mathcal{O}_{\mathbb{P}^n}(d_i)\otimes \mathcal{I}_Z\right)=H^0\left(\mathbb{P}^n, \mathcal{I}_Z(d_i)\right) \quad \textrm{and}\quad
S_Z^{\mathbf{d}}:= \bigoplus_{i=1}^c S_Z^{d_i}.
\end{equation}
In particular, for $x\in \mathbb{P}^n$,  $S_x^{d_i}:=H^0\left(\mathbb{P}^n,\mathcal{O}_{\mathbb{P}^n}(d_i)\otimes \mathcal{I}_x\right)$ will denote the subspace of sections vanishing on $x$, and we will set
$$
S_x^{\d}:= \bigoplus_{i=1}^c S_x^{d_i}. 
$$
For any integer $d_i\geq 1$, we will consider the kernel of the evaluation map
\begin{equation}\label{eq:M_d_i}
0\longrightarrow M_{d_i}\longrightarrow S^{d_i}\otimes \mathcal{O}_{\mathbb{P}^n}\longrightarrow \mathcal{O}_{\mathbb{P}^n}(d_i)\longrightarrow 0
\end{equation}
as well as 
\begin{equation}\label{eq:M_d}
0\longrightarrow M_{\d}\longrightarrow S^{\d}\otimes \mathcal{O}_{\mathbb{P}^n}\longrightarrow \bigoplus_{i=1}^c \mathcal{O}_{\mathbb{P}^n}(d_i)\longrightarrow 0.
\end{equation}
Notice that we have
$$
M_{\d} = \bigoplus_{i=1}^c M_{d_i}.
$$

Throughout the paper we will consider the universal complete intersection ${\mathcal X}\subset {\mathbb P^n}\times S^\d$ of multidegree $\d=(d_1,\ldots,d_c)$, and $X_\F\subset {\mathbb P^n}$ will denote the fibre of ${\mathcal X}$ over $\F\in S^\d$, i.e. the complete intersection  defined by the homogeneous polynomials $\F=(F_1,\ldots, F_c)$. 

Let $U\longrightarrow S^\d$ be an \'etale map, let ${\mathcal X}_U\longrightarrow U$ be the family induced by the base change, and
let ${\mathcal Y}\subset {\mathcal X}_U$ be a universal, reduced and 
irreducible subscheme of relative dimension $k$.
For our purposes we may obviously assume $\mathcal Y$ to be invariant under some 
lift of the natural action of $\textrm{GL}(n+1)$ on ${\mathbb P^n}\times S^\d$:
for any $(x,\F)\in {\mathcal Y}$ and $g\in \textrm{GL}(n+1)$, we set
$$
g(x,\F)= (g(x), \left(g^{-1})^*\F\right)\coloneqq \left((g^{-1})^*F_1,\ldots ,(g^{-1})^*F_c\right).
$$
Let $\widetilde {\mathcal Y}\longrightarrow {\mathcal Y}$ be a desingularization and let $\widetilde{\mathcal Y}\stackrel{j}{\longrightarrow}{\mathcal X}_U$ be the natural induced map. In the following, by abuse of notation, we will  omit the \'etale base change in order to make the notation less heavy.

\subsection{The vertical tangent space}\label{subs:vert}
Let $\pi\colon{\mathcal X}\longrightarrow {\mathbb P^n}$ be the projection on the first component and let $T^\textrm{vert}_{\mathcal X}$ (resp. $T^\textrm{vert}_{\mathcal Y}$) be the vertical part of $T_{\mathcal X}$ (resp. $T_{\mathcal Y}$) with respect to $\pi$, i.e. $T^\textrm{vert}_{\mathcal X}$ (resp. $T^\textrm{vert}_{\mathcal Y}$) is the sheaf defined by
$$
 0\longrightarrow T^\textrm{vert}_{\mathcal X}\longrightarrow T_{\mathcal X}{\buildrel {\pi_*}\over
 \longrightarrow} T_{\mathbb P^n}\longrightarrow 0 \quad ({\text{resp.}}\ \ \  0\longrightarrow T^\textrm{vert}_{\mathcal Y}\longrightarrow 
 T_{\mathcal Y}\ {\buildrel {\pi_*}\over \longrightarrow} T_{\mathbb P^n}\longrightarrow 0).
$$
Notice that from the surjectivity of the map $T_{\mathcal Y}\ {\buildrel {\pi_*}\over 
\longrightarrow} T_{\mathbb P^n}$, we deduce that 
$$
 \codim_{\ T^\textrm{vert}_{{\mathcal X},(y,F)}}\ T^\textrm{vert}_{{\mathcal Y},(y,F)}
 =\codim_{\mathcal X}{\mathcal Y}=n-k-c.
$$
By abuse of notation, we denote by $N_{\widetilde{\mathcal Y}/{\mathcal X}}$ the normal bundle to the map $j\colon\widetilde{\mathcal Y}\to {\mathcal X}$.
Notice also that we have
\begin{equation}\label{eq:normal}
   0\longrightarrow  T^\textrm{vert}_{\widetilde{\mathcal Y}}\longrightarrow j^*T^\textrm{vert}_{{\mathcal X}}\longrightarrow N_{\widetilde{\mathcal Y}/{\mathcal X}}\longrightarrow 0. 
\end{equation}
From the inclusion ${\mathcal X}\hookrightarrow {\mathbb P^n}\times S^\d$, we get the exact sequence
$$
 0\longrightarrow {T_{\mathcal X}}_{|X_\F}\longrightarrow 
 {T_{\mathbb P^n}}_{|X_\F}\oplus (S^\d\otimes {\mathcal O}_{X_\F})
 \longrightarrow \bigoplus_{i=1}^c{\mathcal O}_{X_\F}(d_i)\longrightarrow 0,
$$
which combined with (\ref{eq:M_d}) gives
\begin{equation}\label{eq:M-Tvert}
  0\longrightarrow {M_\d}_{|X_\F}\longrightarrow {T_{\mathcal X}}_{|X_\F}
 \longrightarrow {T_{\mathbb P^n}}_{|X_\F}\longrightarrow 0.
\end{equation}
 In other words 
${M_\d}_{|X_\F}$ identifies to the 
vertical part of $T_{\mathcal X}\otimes {\mathcal O}_{X_\F}$ with respect to the 
projection onto ${\mathbb P^n}$.

Another important consequence of the $\textrm{GL}(n+1)$-invariance of $\mathcal Y$ is the following. 
\begin{lemma}\label{lem:jac}
  At any point $(y, \F)\in \mathcal Y$ with $\F=(F_1,\ldots,F_c)$, the vertical tangent space   
  $
 T^{\textrm{vert}}_{{\mathcal Y},(y,\F)}$ contains $\big\langle  S^1_y \cdot J_\F^{d-1}, \F \big\rangle $, 
where $J_\F^{\d-1}\coloneqq\bigoplus_{i=1}^c J_{F_i}^{d_i-1}$ and $J_{F_i}^{d_i-1}$ is the homogeneous Jacobian ideal of $F_i$ in degree $d_i-1$, i.e. the homogeneous ideal generated by all partial derivatives of $F_i$.
\end{lemma}
\begin{proof}
As $\mathcal Y$ is $\textrm{GL}(n+1)$-invariant, its tangent space (resp. its vertical tangent space) contains the tangent space to the $\textrm{GL}(n+1)$-orbit of the point $(y, \F)$ (resp. the intersection of this orbit with $T^{\textrm{vert}}_{{\mathcal X},(y,\F)}= S^\d_y$. 
Since $\textrm{GL}(n+1)$ acts on each summand of $S^\d=\bigoplus_i S^{d_i}$, it is sufficient to understand the tangent space to the orbit on each summand. 
By \cite[Remarque 18.16]{Vbook} the tangent space to the $\textrm{GL}(n+1)$-orbit of a polynomial $F_i$ is the homogeneous Jacobian ideal $J^{d_i}_{F_i}$ of $F_i$ in degree $d_i$, whose intersection with the vertical part $S^{d_i}_y$ is 
$\big\langle  S^1_y \cdot J_{F_i}^{d_i-1}, F_i \big\rangle$ and the conclusion follows.
\end{proof}
We now prove  a variant of the above that will be useful in the proof of our main theorem. 
As before, we consider a point $(y, \F)\in \mathcal Y$, with $\F=(F_1,\dots,F_c)$, and a line $\ell \ni y$. Without loss of generality we may choose coordinates so that $y=[1:0:\ldots:0]$ and $\ell =\{ X_2=\dots=X_n=0\}$. 
Let $G_{(y,\ell)}\subset \textrm{GL}(n+1)$ be the stabilizer of $(y,\ell)$.
We have the following.    \begin{lemma}\label{lem:jacG1} 
  In the above notation, if   $\mathcal Z\subset \mathcal{Y}$ is a $G_{(y,\ell)}$-invariant subvariety, then  the vertical tangent space   
  $T^{\mathrm{vert}}_{{\mathcal Z},(y,\F)}$ contains $\big\langle  S^1_y  \cdot \frac{\partial \F}{\partial X_0} , S^1_y  \cdot \frac{\partial \F}{\partial X_1} ,S^1_\ell  \cdot \frac{\partial \F}{\partial X_2},\ldots, S^1_\ell  \cdot \frac{\partial \F}{\partial X_n} , \F \big\rangle $, 
where $\frac{\partial \F}{\partial X_i}:= \left(\frac{\partial F_1}{\partial X_i},\ldots, \frac{\partial F_c}{\partial X_i}\right)$.
\end{lemma}
\begin{proof}
    We proceed as in \cite[Remarque 18.16]{Vbook}. First we compute the tangent space to the $G_{(y,\ell)}$-orbit of a homogeneous polynomial $F\in S^d$ in $(n+1)$-variables. It is generated by $\left.\frac{dF_t}{dt}\right|_{t=0}$ where 
    $$F_t=g_t^* F,\quad g_t= I_{n+1}+t A,\quad 0<t\ll 1,
    $$
    and, by the above choice of coordinates, we have
    $$  
    A= 
    \begin{pmatrix}
a_{00} & a_{01} & a_{02}& \cdots & a_{0n}\\
0 & a_{11} & a_{12}& \cdots & a_{1n}\\
0 & 0 & a_{22}& \cdots & a_{2n}\\
\vdots& \vdots& \vdots& \cdots & \vdots&\\
0 & 0 & a_{n2}& \cdots & a_{nn}
\end{pmatrix}.
    $$
    Since
    $$
    g_t^* F(X_0,\ldots, X_n)= F(X_0+t A_0, X_1+tA_1, \ldots, X_n+tA_n),\textrm{ with } A_i=\sum_{j=0}^n a_{ij}X_j, 
    $$
    by computing the derivative we get 
    $$
    \left.\frac{dF_t}{dt}\right|_{t=0}=\left.\frac{d}{dt}\Big(F(X_0+t A_0, X_1+tA_1, \ldots, X_n+tA_n)\Big)\right|_{t=0}=A_0\frac{\partial F}{\partial X} + A_1\frac{\partial F}{\partial X_1}+\ldots +A_n\frac{\partial F}{\partial X_n}. 
    $$
    Since $A_1\in S^1_y$ and $A_i\in S^1_\ell$ for all $i\geq 2$, we deduce that the tangent space to the $G_{(y,\ell)}$-orbit of $F$ equals 
    
    $$\big\langle S^1  \cdot \frac{\partial F}{\partial X_0} , S^1_y  \cdot \frac{\partial F}{\partial X_1} ,S^1_\ell  \cdot \frac{\partial F}{\partial X_2},\ldots, S^1_\ell  \cdot \frac{\partial F}{\partial X_n} , F\, \big\rangle. $$
    As in Lemma \ref{lem:jac} the conclusion now follows by repeating this argument on each summand of $S^\d=\oplus S^{d_i}$ and restricting to the vertical tangent space $S^\d_y=\oplus S^{d_i}_y$.
\end{proof}

As observed by Voisin, a  key point to recover by adjuction global sections of the canonical bundle of a subvariety (see Remark \ref{subs:nonfunziona} for more details) is the following positivity result.
\begin{lemma}[see e.g. {\cite[Proposition 2.2, item (i)]{P03}}]\label{lem:gg}
    For any integer $r>0$ the vector bundle $M_r\otimes {\mathcal O}_{\mathbb P^n}(1)$ is generated by its global sections. 
\end{lemma}

\begin{remark}\label{subs:nonfunziona}
    As explained in details in \cite[Section 1]{V1} (see also \cite[Section 2.2]{P03}) Ein's result can be deduced from Lemma \ref{lem:gg} using the $\GL(n+1)$-invariance of $\mathcal Y\subset \mathcal X\subset \mathbb P^n\times S^{\bf d}$ and adjunction since, using the same notation as above, we have a morphism
\begin{equation*}
    H^0(\wedge^{n-c-k}M_{{\bf d}|X_{\bf F}}\otimes \omega_{X_{\bf F}})\subset H^0(\wedge^{n-c-k}T_{{\mathcal X}_|X_{\bf F}}\otimes \omega_{X_{\bf F}})\cong H^0({\Omega^{N+k}_{\mathcal X}}_{|X_{\bf F}})\to H^0({\Omega^{N+k}_{\tilde {\mathcal Y}}}_{|\tilde Y_{\bf F}}) \cong H^0(\omega_{\tilde Y_{\bf F}}). 
\end{equation*}
Hence as soon as the vector bundle 
\begin{equation}
    \wedge^{n-c-k}M_{{\bf d}|X_{\bf F}}\otimes \omega_{X_{\bf F}}=\wedge^{n-c-k}M_{{\bf d}|X_{\bf F}}\otimes \mathcal O_{X_{\bf F}}(d-n-1)
\end{equation}
is generated by its global sections, which happens for any  $d\geq 2n+1-c-k$ by Lemma \ref{lem:gg}, we deduce from the adjunction map 
\begin{equation}\label{eq:adj}
    H^0\left(\wedge^{n-c-k}M_{{\bf d}|X_{\bf F}}\otimes \omega_{X_{\bf F}}\right)\longrightarrow H^0(\omega_{\tilde Y_{\bf F}})
\end{equation}
that $h^0(\omega_{\tilde Y_{\bf F}})>0$.
To improve Ein's bound by one in the hypersurface case, Voisin studies the base locus of $H^0(\wedge^2 M_d \otimes \mathcal O(1))$ viewed as a space of
sections of a line bundle on the Grassmannian of codimension two subspaces of a fiber of $(T^\textrm{vert}_{\mathcal X})_{|X_F}$. To do so, she exhibited in \cite[Lemma 2.3]{V1} explicit elements in the image of the evaluation map $H^0(\wedge^2 M_d \otimes \mathcal O(1))\to \wedge^2 M_d \otimes \mathcal O(1)_{|x}$, at any point $x\in X_F$, and used them to prove that if the adjunction map (\ref{eq:adj}) fails to provide non-zero  global sections of $\omega_{\tilde Y_{\bf F}}$ then at any (smooth) point the tangent space $T_{\mathcal{Y}, (x,F)}$ must contain the ideal of a line passing through $x$. To illustrate the difficulties that we encounter if we try to do the same for $c\geq 2$ notice that in this case
$$H^0\left(\wedge^2 M_{\bf{d}}(1)\right)\cong \bigoplus_{i=1}^c H^0\left(\wedge^2 M_{d_i}(1)\right) \oplus \bigoplus_{i<j}H^0\left(M_{d_i}\otimes M_{d_j}(1)\right)
$$
and one sees that the wedge products of the sections produced by Voisin are not enough to get a similar conclusion. To make this approach work one has then to further study $H^0\left(M_{d_i}\otimes M_{d_j}(1)\right)$.  In order to deal with a wider range of degrees we prefer to proceed more generally as in  Section \ref{sec:CVmaps}.
\end{remark}

\section{(Bi)contact loci}\label{s:contact}

Let $\d = (d_1,\dots , d_c)$ be a $c$-tuple of integers $2 \leq d_1 \leq \dots \leq d_c$, and let us fix an integer $1\leq r\leq d_c$ and a multi-polynomial $\F=(F_1,\dots,F_c)\in S^{\bf d}$. 
In the proof of our results, we will naturally encounter the following subloci of the complete intersection $X_{\F}\subset \bP^n$, the \textit{contact locus}
$$
\Delta_{r,\bf F}:=\left\{x\in X_{\bf F}\left| \exists \textrm{ a line }\ell\ni x \textrm{ such that } \ell\cdot V(F_i)\geq rx \text{ if }d_i=r \text{ and } \ell\subset V(F_i)\text{ otherwise}\right.\right\},
$$
and the \textit{bicontact locus}
$$
\Lambda_{r,\F}:=\left\{x\in X_\F\left|\begin{array}{l}
     \exists \textrm{ a line }\ell \ni x \textrm{ such that }\ell \subset V(F_i)\textrm{ if } d_i<r, \textrm{ and}\\ \ell \cdot V(F_i) \geq r x +(d_i-r) y\textrm{ for some }y\in \ell\textrm{ if } d_i \geq r    \end{array} \right.\right\}.
$$

In this section, we study some varieties dominating these loci. In particular, since they are described  as zero loci of  generic sections of globally generated vector bundles, these varieties have the advantage of being smooth and their canonical divisors are easy to compute. 

\subsection{Contact loci}\label{sub:contact}

Let $\mathbb G:= \textrm{Gr}(1,n)$ the Grassmannian of projective lines in $\mathbb P^n$.
Let  
$$
\mathcal P\coloneqq\{(x,[\ell])|x\in \ell\}\subset \mathbb P^n \times \mathbb G
$$
be the incidence variety, endowed with the projections $p\colon \mathcal P\longrightarrow \bP^n$ and $q\colon \mathcal P\longrightarrow \mathbb G$.
For any integer $m>0$, let ${\mathcal E}_m$ be the $m^{th}$-symmetric power of the dual of the 
tautological subbundle on $\mathbb G$, 
and recall that, by definition, its fibre at a point $[\ell]$ 
is given by $H^0 ({\ell},{\mathcal O}_{\ell} (m))$,
and its first Chern class is
\begin{equation}\label{eq:c1E-contact}
  c_1({\mathcal E}_m)=\frac{m(m+1)}{2} L.  
\end{equation}

Let ${\mathcal L}_m \coloneqq mL-mH$ be the rank $1$ subbundle of $q^* {\mathcal E}_m$, where 
$$
H\coloneqq p^* \mathcal O_{\mathbb P^n}(1) \quad \textrm{ and } \quad L\coloneqq q^* \mathcal O_{\mathbb G}(1)
. $$ 
We note that its fibre ${\mathcal L}_{m,(x,[\ell])}$ 
is equal to the space of degree $m$ 
homogeneous polynomials on $\ell$ vanishing to the order $m$ at $x$. 
Finally, let ${\mathcal F}_m$ be the quotient 
\begin{eqnarray}\label{L,E,F}
 0\longrightarrow {\mathcal L}_m\longrightarrow q^* {\mathcal E}_m 
 \longrightarrow {\mathcal F}_m\longrightarrow 0.
\end{eqnarray}
We then have that 
\begin{equation}\label{eq:c1F-contact}
    c_1(\mathcal F_m)=\frac{m(m-1)}{2}L+mH.
\end{equation}
It is possible to associate to every homogeneous polynomial $F\in S^m$ a section 
$\sigma_F\in H^0 (G,{\mathcal E}_m)$, whose value at a point 
$[\ell]$ is exactly the polynomial $F_{|\ell}$. 
We will denote by ${ {\tau}}_F$ 
the induced section in $H^0 ({\mathcal P},{\mathcal F}_m)$.

Now let ${\bf d}=(d_1,\ldots,d_c)$ be an $c$-tuple of positive integers and, for any $2\leq d_1\leq r\leq d_c$, consider the vector bundle
$$
\mathcal A_{{\bf d}, r}= \Big( \bigoplus_{d_j\not=r} q^*\mathcal E_{d_j} \Big) \oplus \Big( \bigoplus_{d_i=r} \mathcal F_{d_i} \Big).
$$
For every ${\bf F}=(F_1,\ldots,F_c)\in S^{\bf d}$, we can define a global section of $\mathcal A_{{\bf d}, r}$ by considering
$$
s_{{\bf F}, r} :=\Big( \bigoplus_{d_j\not=r} \sigma_{F_j}  \Big) \oplus \Big( \bigoplus_{d_i=r} \sigma_{F_i} \Big)
$$
and we set
$$
\widetilde\Delta_{r,\bf F}\coloneqq V(s_{{\bf F}, r}). 
$$
Notice that the vector bundle $\mathcal A_{{\bf d}, r}$ is generated by its global sections since all its factors are. 
Therefore, for a generic ${\bf F}\in S^{\bf d}$, the zero locus $
\widetilde\Delta_{r,\bf F}=V(s_{{\bf F}, r})$ is smooth, and its dimension is
\begin{equation}\label{eq:dim Delta_r}
\dim V(s_{{\bf F}, r}) = \dim \mathcal P - \rk (\mathcal A_{{\bf d}, r}) = 2n-1 - \sum_{d_j\not=r} (d_j+1) - \sum_{d_i=r} d_i.     
\end{equation}

\begin{lemma}\label{lem:cancontact} 
    The canonical class of $\widetilde\Delta_{r,\bf F}$ is given by 
$$
K_{\widetilde\Delta_{r,\bf F}}= (\alpha r -2)H+\left( \sum_{j=1}^c \frac{d_j(d_j+1)}{2} - \alpha r -n\right) L, 
$$
where $\alpha\coloneqq\#\{i\,|\,d_i=r\}$.
\end{lemma}
\begin{proof} 
   The formula follows from the adjunction formula using (\ref{eq:c1E-contact}), (\ref{eq:c1F-contact}), and the fact that 
    $K_{\mathcal P}=-2H-nL$ (see e.g. \cite[proof of Proposition 1]{V2}). 
    \end{proof}

\subsection{Bicontact loci}

Given $\d = (d_1,\dots , d_c)$ be a $c$-tuple of integers $2 \leq d_1 \leq \dots \leq d_c$, and an integer $1\leq r\leq d_c$, we set
$$
\beta:=\#\{i\,| \, d_i\geq r\}.
$$
As above, we consider the Grassmannian $\mathbb G:= \textrm{Gr}(1,n)$ of projective lines in $\mathbb P^n$ and the universal line $\mathcal P\subset \mathbb P^n \times \mathbb G$, endowed with the projections $p\colon \mathcal P\longrightarrow \bP^n$ and $q\colon \mathcal P\longrightarrow \mathbb G$.

We distinguish three cases, depending on $r$ and $\d$.

\medskip
{\bf Case 1: $\beta r \geq 3$ and $\sum_{d_i\geq r}(d_i-r)\geq 2$.}

\noindent 
Let ${\mathcal O}_{\mathbb G} (1)$ be the line bundle on $\mathbb G$ which gives the Pl\"ucker embedding.
Let $Z$ be the blow-up along the diagonal $\textrm{Diag}\subset{\mathbb P}^n\times {\mathbb P}^n$, with projections
\begin{equation}
\xymatrix{
{Z:= \textrm{Bl}_{\textrm{Diag}}{\mathbb P}^n\times {\mathbb P}^n}\ar[r]^{\ \ \ \ \ b}
& {\mathbb P}^n \times{\mathbb P}^n\ar[r]^{p_2}\ar[d]^{p_1}
& \mathbb P^n.\\
&\mathbb P^n
& 
& \\
}\end{equation}
Then we have a natural map 
\begin{equation}
f\colon Z \longrightarrow \mathbb G\qquad  z\longmapsto \ell_{z},
\end{equation}
where for general $z=b^{-1}(x,y)\in Z$, $\ell_z$ is the line passing through $x$ and $y$.
For $i=1,2$, let $\widetilde p_i\coloneqq p_i\circ b$ and consider the line bundles on $Z$ defined as
$$H_1\coloneqq\widetilde p_1^* {\mathcal O}_{\mathbb P^n}(1), \quad H_2\coloneqq\widetilde p_2^* {\mathcal O}_{\mathbb P^n}(1), \quad\text{and}\quad L\coloneqq f^* {\mathcal O}_{\mathbb G} (1).$$
The variety $Z$ comes together with a projective bundle
${\mathbb P}{\buildrel \pi \over {\longrightarrow}} Z$, and $\forall m\geq 0$ we define  
${\mathcal E}_m:= \pi_* {\mathcal O}_{\mathbb P} (m)$.
Notice that the fibre 
of ${\mathcal E}_m$ at $z$ is equal to $H^0 ({\ell_z},{\mathcal O}_{\ell_z} (m))$ and that we have  
\begin{equation}\label{eq:c1-bicontact}
      c_1({\mathcal E}_m)=\frac{m(m+1)}{2} L.  
\end{equation}

If $0\leq r\leq m$, we may consider the line bundle ${\mathcal L}_{r,m-r}\subset {\mathcal E}_m$,
whose fibre at $z\in Z$ is given by the one dimensional space of
polynomials $P\in H^0 ({\ell_z},{\mathcal O}_{\ell_z} (m))$ vanishing at $x$ to 
the order $r$ and at $y$ to the order $m-r$, where 
$(x,y)=b(z)\in {\mathbb P}^n \times{\mathbb P}^n$. 
Moreover, we define ${\mathcal G}_{r,m-r}:={\mathcal E}_m/{\mathcal L}_{r,m-r}$. 
To any polynomial 
$F\in S^m$ we can associate a section $\sigma_F\in H^0 (Z,{\mathcal E}_m)$, 
whose value at a point 
$z$ is exactly the polynomial $F_{|\ell_z}\in {{\mathcal E}_m}_{|z}$, 
and we will denote 
by ${{\tau}_F}$ its image 
in $H^0 (Z,{\mathcal G}_{r,m-r})$. 

We now
consider the vector bundle
$$
\mathcal B_{{\bf d}, r}\coloneqq \Big( \bigoplus_{d_i<r} q^*\mathcal E_{d_i}\Big) \oplus\ \Big(\bigoplus_{d_i\geq r}q^*\mathcal G_{r,d_{i}}\Big).
$$
For every ${\bf F}=(F_1,\ldots,F_c)\in S^{\bf d}$, we can define a global section of $\mathcal B_{{\bf d}, r}$ by considering:
$$
s_{{\bf F}, r} :=\Big( \bigoplus_{d_i<r} \sigma_{F_i}\Big)\oplus \Big( \bigoplus_{d_i\geq r} \tau_{F_{i}}\Big)
$$
Then we define 
\begin{equation}\label{eq:Lambda1}
    {\widetilde\Lambda}_{r,\F}:= V({{s}_{{\bf F}, r}}),
\end{equation}
that is 
\begin{equation}\label{eq:Lambda1bis}
{\widetilde\Lambda}_{r, \mathbf{F}}=\overline{\left\{(x,x')\in \mathrm{Bl}_{\textrm{Diag}}(\mathbb{P}^n\times \mathbb{P}^n)\left|\begin{array}{c}\forall i=1,\dots,c \text{ the line }\ell\coloneqq \langle x,x'\rangle \text{ satisfies }\\ \ell\cdot V(F_i)\geq rx+(d_i-r)x' \end{array} \right.\right\}}.
\end{equation}
By construction, we have 
$\widetilde{p_1}({\widetilde\Lambda}_{r,\F})={\Lambda}_{r,\F}$.
Since ${\mathcal B}_{\d,r}$ is generated by its global sections, the variety
${\tilde\Lambda}_{r,\F}$ is smooth of dimension 
\begin{equation}\label{eq:dim-lambda}
    \dim {\tilde\Lambda}_{r,\F} = \dim(Z)-\rk \mathcal B_{{\bf d}, r}= 2n - \sum_{d_i<r} (d_i+1) - \sum_{d_i\geq r} d_i\ .
\end{equation}
Moreover, we compute the canonical bundle of 
${\tilde\Lambda}_{r,F}$.

\begin{lemma}\label{lem:can-bicontact}
   The canonical class of $\widetilde\Lambda_{r,\bf F}$
is given by 
$$
K_{{\widetilde\Lambda}_{r,\F}}= 
 (\beta r-2)H_1 + \left[ -2+ \sum_{d_i\geq r}(d_i-r)\right]H_2 + \left[- n+1+\sum_{d_i<r}\binom{d_i+1}{2} + \sum_{d_i\geq r}\binom{d_i}{2}  \right]L.
$$
\end{lemma}
\begin{proof} 
   By adjunction, 
we have that $ K_{{\widetilde\Lambda}_{r,\F}}= (K_Z)_{|{\widetilde\Lambda}_{r,\F}} + c_1 ({\mathcal B}_{\d,r})$. As remarked in \cite[Proof of Proposition 1]{V2}, the Picard group of $Z$ is generated by $H_1$, $H_2$ and $L$, the canonical class of $Z$ is 
$K_Z= -2H_1 - 2H_2 +(-n+1)L$, and the class of ${\mathcal L}_{(r,m-r),F}$ is given by $rH_1 + (m-r)H_2$.
The formula then follows  using (\ref{eq:c1-bicontact}) and the above.
\end{proof}

\medskip
{\bf Case 2: $d_1,\ldots,d_{c-1}\leq r$ and $d_c=r+1$.} \\
Consider the condition $\sum_{d_i\geq r}(d_i-r)\geq 2$ appearing in Case 1.
We point out that it fails if and only if $d_1,\ldots,d_{c-1}\leq r$ and $d_c\in \{r,r+1\}$.
When $d_c=r$, the locus ${\Lambda}_{r,\F}$ coincide with the contact locus ${\Delta}_{r,\F}$, which has been considered in \S \ref{sub:contact}.
Thus we assume that $d_c=r+1$.

Consider the universal line $\mathcal{P}\subset \mathbb P^n\times  \mathbb G$. 
Let ${\mathbb P}{\buildrel \pi \over {\longrightarrow}} \mathcal{P}$ be the pull-back of the universal $\mathbb P^1$-bundle over $\mathbb G$, and $\tau$ the natural section of $\pi$.
As in \cite[Proof of Proposition 1]{V2}
\begin{equation}\label{eq:cangamma}
    K_\mathcal{P}=-2H-nL,
\end{equation}
where $L$ is the pull-back of $\mathcal O_{\mathbb G}(1)$ and $H:=\tau^* {\mathcal O}_{\mathbb P} (1)$.

For an integer $m\geq 1$ we consider the bundle ${\mathcal E}_m:= \pi_* {\mathcal O}_{\mathbb P} (m)$  
over $\mathcal{P}$. Its first Chern class is given by 
\begin{equation}\label{eq:c1E}
c_1({\mathcal E}_m)=\frac{m(m+1)}{2} L.
\end{equation}
We will consider two subbundles of ${\mathcal E}_m$. 
First consider the line bundle ${\mathcal L}_1$ defined by 
$$
 0\longrightarrow {\mathcal L}_1\longrightarrow {\mathcal E}_1\longrightarrow 
 H\longrightarrow 0.
$$
Its $m$-th tensor power ${\mathcal L}_1^{\otimes m}$ is the line bundle whose fiber at a point $(x,[\ell])$ is the vector space $H^0(\ell, \mathcal O_{\ell}(-mx))$ of degree $m$ homogeneous polynomials on $\ell$ vanishing with maximal order at $x$. Define 
${\mathcal F}_m:= {\mathcal E}_m/{\mathcal L}_m$. 
Hence
\begin{equation}\label{eq:c1F}
c_1({\mathcal F}_m)=c_1({\mathcal E}_m) - c_1({\mathcal L}_m)= \frac{m(m+1)}{2} L - m(c_1({\mathcal E}_1)- H)= \Big( \frac{m(m-1)}{2}\Big)L +mH. 
\end{equation}

We then consider the rank $2$ subbbundle 
${\mathcal K_{m-1}}\subset {\mathcal E}_m$
whose fiber at a point $(x,\ell)$ is given by the
polynomials $P\in H^0 ({\mathcal O}_{\ell}(m))$ vanishing to the order 
at least $(m-1)$ at $x$. 
Then ${\mathcal K_{m-1}}\cong {\mathcal L}_1^{m-1}\otimes {\mathcal E}_1$. 
 By its definition we get
$$
c_1(\mathcal K_{m-1})=c_1(\mathcal E_{1})+2(m-1)c_1(\mathcal L_{1})=c_1(\mathcal E_{1})+2(m-1)\Big[ c_1(\mathcal E_{1}) -H\Big]= (2m-1)L-2(m-1)H.
$$
Let ${\mathcal G}_m$ be the quotient ${\mathcal E}_m/{\mathcal K_{m-1}}$.
We then have 
\begin{equation}\label{eq:c1G}
c_1({\mathcal G}_m) = \Big[ \frac{m(m+1)}{2}-(2m-1)\Big] L + 2(m-1)H= \Big[ \frac{(m-2)(m-1)}{2}\Big] L + 2(m-1)H.
\end{equation}

As before, to any $F\in S^m$ we can associate a global 
section $\sigma_F$ of ${\mathcal E}_m$ and, by restricting it to the quotients, a global section 
$\tau_F$ of ${\mathcal F}_m$ and a global section 
$\rho_F$ of ${\mathcal G}_m$. 
We now consider the globally generated vector bundle 
$$
\mathcal C_{\d,r}=\Big( \bigoplus_{d_i< r} \mathcal E_{d_i}\Big) \oplus \Big( \bigoplus_{d_i= r} \mathcal F_{r}\Big) \oplus  \mathcal G_{r+1}
$$
and, for $\F=(F_1,\ldots,F_c)\in S^\d$, its global section 
$$
s_{{\bf F}, r} :=\Big( \bigoplus_{d_i<r} \sigma_{F_i}\Big)\oplus \Big( \bigoplus_{d_i= r} \tau_{F_{i}}\Big)\oplus  \rho_{F_{c}}.
$$
Then we define 
\begin{equation}\label{eq:Lambda2}
    {\widetilde\Lambda}_{r,\F}:= V({{s}_{{\bf F}, r}}). 
\end{equation}
By construction, we have that 
${p}({\widetilde\Lambda}_{r,\F})={\Lambda}_{r,\F}$.
Since ${\mathcal C}_{\d,r}$ is generated by its global sections, the variety ${\tilde\Lambda}_{r,\F}$ is smooth of dimension $\dim {\tilde\Lambda}_{r,\F} = \dim(\Gamma)-\rk \mathcal C_{{\bf d}, r}$, which yields
\begin{equation}\label{eq:dim-lambda-bis}
    \dim {\tilde\Lambda}_{r,\F} = 2n-1 - \sum_{d_i<r} (d_i+1) - \sum_{d_i= r} r -  r= 2n-1 - \sum_{d_i<r} (d_i+1) - \sum_{d_i\geq  r} r\ .
\end{equation}

\begin{lemma}\label{lem:can-bicontact-bis}
  The canonical class of $\tilde\Lambda_{r,\bf F}$
is given by 
$$
K_{{\tilde\Lambda}_{r,\F}}=
 (-2+\gamma r+2 r)H +  \left[- n+\sum_{d_i<r}\binom{d_i+1}{2} +  \sum_{d_i\geq r}\binom{r}{2}  \right]L,
$$
where $\gamma:=\#\{i:d_i=r\}$. 

\end{lemma}
\begin{proof} 
   By adjunction, 
we have $
 K_{{\tilde\Lambda}_{r,\F}}= (K_\Gamma)_{|{\tilde\Lambda}_{r,\F}} + c_1 ({\mathcal C}_{\d,r})$. Then, from (\ref{eq:cangamma}), (\ref{eq:c1E}), (\ref{eq:c1F}) and (\ref{eq:c1G}) we get
$$
K_{{\tilde\Lambda}_{r,\F}}=
\Big[-2H-nL\Big]+ \Big[ \sum_{d_i<r} \frac{d_i(d_i+1)}{2} L\Big]+ \Big[ \sum_{d_i=r}\frac{r(r-1)}{2}L +rH \Big]+ \Big[\frac{r(r-1)}{2} L + 2rH\Big]
$$ 
and the formula now follows.   
\end{proof}

{\bf Case 3: $\beta r \leq 2$.}\\
We describe here smooth models of bicontact loci generalizing to complete intersections of arbitrary codimension $c$. 
The case  $c=1$ is treated in \cite[Proof of Theorem 5.2]{CoR22} (which in turn was based on \cite{CR04}). 
To simplify the notation we set $S:=S^\d$.
First, we define the incidence locus $\Delta_1\subset \mathbb P^n\times \mathbb G \times S$ as follows. 
The projections onto the three factors of $\mathbb P^n\times \mathbb G \times S$ will be denoted respectively by $p$, $\pi_\bG$ and $\pi_S$ and the  pull-back of the positive generator of $\textrm{Pic}(\bP^n)$ (respectively of $\textrm{Pic}(\bG)$), are denoted by $H$ (resp. $L$). The incidence locus $\Delta_1$ the common zero locus of the two maps
$$
p^*\mathcal O_{\mathbb P^n}(-1)\longrightarrow \pi^*_G \mathcal{Q}
$$
and 
$$
\pi_S^*\mathcal O_S(-1)\longrightarrow
p^*\left( \oplus_{i=1}^c \mathcal O_{\mathbb P^n}(d_i)\right).
$$
Therefore, by adjunction,
\begin{equation}\label{eq:canDelta1}
\begin{aligned}
    \omega_{\Delta_1} & = (\omega_{\mathbb P^n\times \mathbb G \times S})_{|\Delta_1}\otimes c_1 \left[(p^*\mathcal O_{\mathbb P^n}(1)\otimes \pi_G^*\mathcal Q)\oplus \left( \pi_S^*\mathcal O_S(1)\otimes p^*( \oplus_{i=1}^c \mathcal O_{\mathbb P^n}(d_i))\right)\right]\\
    & = -nL+(d-2)H.
\end{aligned}
\end{equation}
Then, for any integer $r\geq 1$, we consider
$$
\Delta_{r} :=\{(x,\ell,\F)| \ell\cdot V(F_i)\geq rx\}
\subset \mathbb P^n\times \mathbb G \times S$$
inside $\Delta_1$. Notice that if $r>d_i$ the condition $\ell\cdot V(F_i)\geq rx$ implies that the line $\ell$ is contained in the hypersurface $V(F_i)$.
Similarly, if $\s=(s_1,\ldots, s_c)$ is a $c$-tuple of integers, we define  
\begin{equation}\label{eq:Delta_r,s}
\Delta_{r,\s} :=\{(x,y,\ell,\F)| \ell\cdot V(F_i)\geq rx+s_i y\}
\subset \mathbb P^n\times \mathbb P^n\times \mathbb G \times S.
\end{equation}
We denote by   $\Delta_{r,\s, \F}$ the fiber of $\Delta_{r,\s}$ over $\F\in S^\d$. 
The projections onto the four factors of $\mathbb P^n\times \mathbb P^n\times \mathbb G \times S$ will be denoted respectively by $p_1$, $p_2,$ $\pi_\bG$ and $\pi_S$ and the two pull-backs of the positive generator of $\textrm{Pic}(\bP^n)$, respectively of $\textrm{Pic}(\bG)$, are denoted by $H_1$ and $H_2$ (resp. $L$).
  If $\mathcal S$ is the tautological subbundle on $\bG$, the natural injection 
$$
0\longrightarrow p^*\mathcal O_{\bP^n}(-1) \longrightarrow \pi^*_\bG\mathcal S
$$
yields a rank one quotient $R^\vee$, which by definition satisfies
$$
R=\pi^*_\bG \mathcal O_{\bG}(1)\otimes p^*\mathcal O_{\bP^n}(-1).
$$
In other words, the fiber of $R$ at a point $(x,\ell, F)$ is $H^0(\ell, \mathcal O_\ell(1)(-x))$.
For any integer $m\geq 1$ we can define a natural filtration $\mathcal F^\bullet_m$ on $\pi_\bG^* Sym^m \mathcal S^\vee$ by homogeneous degree $m$ polynomials vanishing to the order $j$ at $x$ modulo those vanishing at higher order. More explicitly 
$$
\mathcal F^j_m/\mathcal F_m^{j+1}=p^*\mathcal O_{\bP^n}(m-j)\otimes R^j= p^*\mathcal O_{\bP^n}(m-2j)\otimes \pi^*_\bG\mathcal O_{\bG}(j).
$$
Notice that from this formula we deduce that if $m\geq r$, then
\begin{equation}\label{eq:c_1}
    c_1\left(\mathcal F^1_m/\mathcal F_m^{r}\right)=c_1\left(\bigotimes_{j=1}^{r-1} p^*\mathcal O_{\bP^n}(m-2j)\otimes \pi^*_\bG\mathcal O_{\bG}(j)\right)=(r-1)(m-r) H+\frac{r(r-1)}{2}L.
\end{equation}
If instead $m<r$, then 
\begin{equation}\label{eq:c_1 bis}
c_1\left(\mathcal F^1_m/\mathcal F_m^{r}\right)=c_1\left(\mathcal F^1_m/\mathcal F_m^{m+1}\right)=-m H+\frac{m(m+1)}{2}L.
\end{equation}

On $S$ we consider the tautological line bundle $\mathcal O_S(-1)$ whose fiber at $\F\in S$ is $\mathbb C\cdot \F$.
Now consider, for any $m\geq 1$, the maps
$$
\varphi_m: \pi_S^* \mathcal O_S(-1)\to \mathcal F^1_m/ \mathcal F^r_m
$$
and, for a $c$-tuple of integers $\d=(d_1,\ldots, d_c)$ the map
$$
\varphi_{r,\d}\coloneqq (\varphi_{d_1},\ldots, \varphi_{d_c})\colon \pi_S^* \mathcal O_S(-1)\longrightarrow \oplus_{i=1}^c\mathcal F^1_{d_i} / \mathcal F^r_{d_i}
$$
Then 
$\Delta_r$ is the zero scheme of $\varphi_{r,\d}$ and therefore, by adjunction and using the fact that $\mathcal O_S(-1)$ is trivial, from (\ref{eq:c_1}) and (\ref{eq:c_1 bis}) we deduce that 
\begin{align*}
    \omega_{\Delta_r} & = \omega_{\Delta_1}\otimes c_1(\oplus_{i=1}^c\mathcal F^1_{d_i} / \mathcal F^r_{d_i})\\
    &= \left(-n + \sum_{d_i<r}\frac{d_i(d_i+1)}{2} + \sum_{d_i\geq r}\frac{r(r-1)}{2} \right) L + \left(-2 + \sum_{d_i\geq r} r(d_i-r+1)\right) H.
\end{align*}
We now do similar computations for $\Delta_{r,\s}$. To avoid cumbersome notation we still denote by $R$ its natural pull-back to $\mathbb P^n\times \mathbb P^n\times \mathbb G \times S$.
We first define 
$$
q\colon p_2^*\mathcal O_{\bP^n}(-1)\longrightarrow \pi_\bG^*\mathcal Q
$$
where $\mathcal Q$ is the quotient bundle on the Grassmannian $\bG$. For any $c$-tuple of integers $\d=(d_1,\ldots,d_c)$ we also have a map:
$$
\pi_S^*\mathcal O_S(-1)\longrightarrow \oplus_{i=1}^c\  p_2^* \mathcal O_{\bP^n}(d_i).
$$
The locus $\Delta_{r,\bf 1}$, where ${\bf 1}$ is simply the $c$-tuple $(1,\ldots,1)$, is the common zero scheme of these two maps on the preimage of $\Delta_r$ in 
$\mathbb P^n\times \mathbb P^n\times \mathbb G \times S$ (which we still denote by $\Delta_r$).
Thus
\begin{align*}
\omega_{\Delta_{r,\bf 1}} & = \omega_{\Delta_{r}} \otimes p_2^* \omega_{\bP^n} \otimes c_1\left(p_2^*\mathcal O_{\bP^n}(1)\otimes \pi_\bG^*\mathcal Q\right) \otimes c_1\left(\oplus_{i=1}^c\  p_2^* \mathcal O_{\bP^n}(d_i)\right)\\
& = \Big(-n  + 1+ \sum_{d_i<r}\frac{d_i(d_i+1)}{2} + \sum_{d_i\geq r}\frac{r(r-1)}{2} \Big) L + \Big(-2 + \sum_{d_i\geq r} r(d_i-r+1)\Big) H_1 + (d-2) H_2.    
\end{align*}
We can define a filtration $\mathcal G^\bullet_m$ on $R^r\otimes \pi_\bG^* \textrm{Sym}^{m-r} \mathcal S^\vee$ in a way similar to the definition of $\mathcal F^\bullet_m$ by the order of vanishing at the point $y$, so that 
$$
\mathcal G^j_m/\mathcal G^{j+1}_m = p_2^*\mathcal O_{\bP^n}(m-2j)\otimes \pi^*_\bG\mathcal O_{\bG}(j).
$$
Notice that from this formula we deduce that if $m\geq s$, then
\begin{equation}\label{eq:c_1 ter}
    c_1\left(\mathcal G^1_m/\mathcal G_m^{s}\right)=c_1\left(\bigotimes_{j=1}^{s-1} p_2^*\mathcal O_{\bP^n}(m-2j)\otimes \pi^*_\bG\mathcal O_{\bG}(j)\right)=(s-1)(m-s) H_2+\frac{s(s-1)}{2}L.
\end{equation}
Now we will apply the above for two $c$-tuples of integers $\d=(d_1,\ldots, d_c)$  
and $\s=(s_1,\ldots, s_c)$ with $s_i\coloneqq \max\{0,d_i-r\}$, and 
the map
$$
\psi_{r,\d, \s}\colon 
 \pi_S^* \mathcal O_S(-1)\longrightarrow \oplus_{d_i> r}\mathcal G^1_{d_i} / \mathcal G^{s_i}_{d_i}.
$$
Then 
$\Delta_{r,\s}$ is the zero scheme of $\psi_{r,\d,\s}$ on $\Delta_{r,\bf 1}$.  
We have the following.

\begin{lemma}\label{lem:can-bicontact-ter}
The canonical class of $\Delta_{r,\s,\F}$
is given by 
\begin{align*}
\omega_{\Delta_{r,\s,\F}} 
& = \left(-n + 1+ \sum_{d_i<r}\frac{d_i(d_i+1)}{2} + \sum_{d_i\geq r}\frac{r(r-1)}{2} + \sum_{d_i>r}\frac{s_i(s_i-1)}{2} \right) L +\\
& + \left(-2 + \sum_{d_i\geq r} r(d_i-r+1)\right) H_1 + \left(-2 + \sum_{d_i> r} s_i(d_i-s_i+1)\right) H_2.    
\end{align*}
Moreover, for general $\F\in S^\d$ and $\s=\d-r\coloneqq (d_1-r,\dots,d_c-r)$, 
  $\Delta_{r,\s, \F}$ is birational to the scheme $\widetilde \Lambda_{r,\F}$ defined in \eqref{eq:Lambda1bis}.
\end{lemma}
\begin{proof}
    Since 
$\Delta_{r,\s}$ is the zero scheme of $\psi_{r,\d,\s}$ on $\Delta_{r,\bf 1}$  
  by adjunction we have that
$$
\omega_{\Delta_{r,\s}} = \omega_{\Delta_{r,\bf 1}} \otimes c_1\left(\oplus_{d_i> r}\mathcal G^1_{d_i} / \mathcal G^{s_i}_{d_i}\right).
$$
Now the formula for its canonical class follows from  (\ref{eq:c_1 ter}). 

For the second part of the statement, we have an obvious birational map $\Delta_{r,\d-r,\F}\dashrightarrow \widetilde \Lambda_{r,\F}$ defined outside the diagonal of $\bP^n\times \bP^n$, which sends $(x,y,[\ell])$ to $(x,y)$.  
\end{proof}

\section{Multiplication maps}\label{sec:CVmaps}
We have recalled in Section \ref{subs:nonfunziona} that the difficulty of extending Ein's result to lower degrees lies in the lack of global generation of the vector bundles $\bigwedge^r M_{\bf{d}}(s)$, when $s<r$.
The goal of the section is to 
systematically study this lack of positivity, and deduce algebraic information on the subspaces lying in the base locus, namely the fact that any such a subspace must have an unusually large intersection with the ideal of a line.  
To this aim, we extend to any $c\geq 1$ the approach of \cite[\S 2.2]{C03} and \cite[\S 3]{CR04} for the case $c=1$, which in turn  generalizes \cite[Lemma 2.3]{V1}, where the base locus was first studied for $c=1,r=2$  and $s=1$ (see also \cite[Lemma 3.2]{P03}, where the case $c=1,r=4,s=2$ appeared). 

\smallskip
Let us consider a point $x\in \mathbb{P}^n$, and the space $S_x^{\mathbf{d}}$ of $c$-tuples $\mathbf{Q}=(Q_1,\dots, Q_c)$, where any $Q_j$ is a homogeneous polynomial of degree $d_j$ vanishing at $x$. 
Let $T\subset S_x^{\mathbf{d}}$ be a subspace of codimension $t:=\codim T$.
Consider the vector space
$$S^{\mathbf{d}-\mathbf{1}}:= \bigoplus_{i=1}^c S^{d_i-1},$$
and let $\mathbf{P}_1,\dots,\mathbf{P}_s\in S^{\mathbf{d}-\mathbf{1}}$ be general elements.
Then we define the linear map
\begin{equation}\label{eq:mu}
\mu_s=\mu_{s,T}\colon \left(S^1_x\right)^{\oplus s}\longrightarrow \frac{S_x^{\mathbf{d}}}{T},
\end{equation}
which sends a $s$-tuple $(L_1,\dots,L_s)$ of linear polynomials vanishing at $x$ to the vector 
$$\mu_s\left(L_1,\dots,L_s\right):=L_1 \cdot \mathbf{P}_1+\dots + L_s\cdot \mathbf{P}_s \mod T,$$
where $L_j \cdot \mathbf{P}_j$ denotes the multiplication by $L_j$ of any component of $\mathbf{P}_j$. 
Moreover, for any $s\in \mathbb{N}$, we set
$$
\gamma(s):=\left\{\begin{array}{ll} 0 &\text{if }s=0\\ \dim(\operatorname{Im} \mu_s) & \text{otherwise.}\end{array}\right.
$$
\begin{remark}\label{rem:monotonia}
Since $\mathbf{P}_1,\dots,\mathbf{P}_s\in S^{\mathbf{d}-\mathbf{1}}$ are general, the following are straightforward:
\begin{itemize}
  \item[(i)] For any $s\geq 0$, the function $\gamma(s)$ is strictly increasing until $\mu_s$ is surjective. 
  In particular, since $\dim \big( S_x^{\mathbf{d}}/T\big)=t$, the map $\mu_t$ is surjective.
  \item[(ii)] The function $\gamma(s+1)-\gamma(s)$ is non-increasing, and it vanishes if and only if $\mu_s$ is surjective.
\end{itemize}
\end{remark}
In view of Remark \ref{rem:monotonia}, we define the integers
\begin{equation}\label{eq:s&s'}
    s_T:=\min \left\{\left. s\in \mathbb{Z}_{>0}\right|\mu_s \text{ is surjective}\right\} \quad\text{and}\quad s'_T:=\min \left\{\left. s\in \mathbb{N}\right|\gamma(s+1)-\gamma(s)\leq 1\right\}.
\end{equation}
Hence $s'_T\leq s_T\leq t$ and 
\begin{equation}\label{eq:T}
t=\gamma(s'_T)+(s_T-s'_T)\geq 2s'_T+(s_T-s'_T)=s'_T+s_T.    
\end{equation} 
In particular, if $s'_T=s_T$ we deduce that
\begin{equation}\label{eq:s'=s}
s_T\leq t/2. 
\end{equation}  
On the other hand, when $s_T-s'_T\geq 2$ the following holds.

\begin{lemma}\label{lem:retta}
Using notation as above, we assume that $s_T-s'_T\geq 2$.
\begin{itemize}
  \item[(i)] There exists a unique line $\ell=\ell_T\subset \mathbb{P}^n$ passing through $x$ such that 
  \begin{equation}\label{eq:ideale^d}
  S_{\ell}^{\mathbf{d}}\subseteq T + \sum_{i=1}^{s'_T}  S^1_x \cdot \mathbf{P}_i 
  \end{equation}
  (for the definition of $S_\ell^{\bf d}$ see \eqref{eq:S_Z}).
  \item[(ii)] The line does not dependent on the choice of the general elements $\mathbf{P}_1,\dots,\mathbf{P}_{s'_T}\in S^{\mathbf{d}-\mathbf{1}}$.
  \item[(iii)] We have
  $$
  \dim\frac{T}{T\cap S_{\ell}^{\mathbf{d}}}=\dim\frac{T+ S_{\ell}^{\mathbf{d}}}{S_{\ell}^{\mathbf{d}}}=d-s_T\quad\text{and}\quad  
  \dim\frac{S_{\ell}^{\mathbf{d}}}{T\cap S_{\ell}^{\mathbf{d}}}=\dim\frac{T+ S_{\ell}^{\mathbf{d}}}{T}=\gamma(s'_T)-s'_T.
  $$
\end{itemize}

\begin{proof}
(i) and (ii). 
Let $\mathbf{P}\in S^{\mathbf{d}-\mathbf{1}}$ be a general element. Since $s_T-s'_T\geq 2$, by Remark \ref{rem:monotonia}, item (ii), and the definition of $s'_T$ we have that $\gamma(s'_T+1)-\gamma(s'_T)=1$.
Therefore, 
\begin{equation}\label{eq:n-1}
\dim\left(S^1_x\cdot \mathbf{P} \cap \left(T+ \sum_{i=1}^{s'_T}  S^1_x \cdot \mathbf{P}_i  \right)\right)=\dim (S^1_x \cdot \mathbf{P}) -1=n-1. 
\end{equation}
Hence there exists a basis $\{L_1,\dots,L_n\}$ of  $S_{x}^{1}$ such that 
$$
\langle L_1 \cdot \mathbf{P},\dots, L_{n-1}\cdot \mathbf{P}\rangle \subseteq T + \sum_{i=1}^{s'_T}  S^1_x \cdot \mathbf{P}_i \quad \text{and} \quad L_n\cdot \mathbf{P}\not\in T +  \sum_{i=i}^{s'_T}  S^1_x \cdot \mathbf{P}_i.
$$
We define the line $\ell=\ell_T\subset \mathbb{P}^n$ for which (\ref{eq:ideale^d}) will hold by the vanishing of the linear forms $L_1,\dots,L_{n-1}$, so that $S^1_{\ell}\cdot \mathbf{P}=\langle L_1\cdot \mathbf{P},\dots, L_{n-1}\cdot \mathbf{P}\rangle$. 
As $\gamma(s'_T+2)-\gamma(s'_T+1)=1$, we also have
\begin{equation}\label{eq:n-1 bis}
\dim\left(S^1_x\cdot \mathbf{P} \cap \left(T+ \sum_{i=1}^{s'_T+1}  S^1_x\cdot  \mathbf{P}_i  \right)\right)=n-1.
\end{equation}
In particular, the $(n-1)$-dimensional subspaces in \eqref{eq:n-1} and \eqref{eq:n-1 bis} coincide.
Therefore, the line $\ell_T$ does not depend on $\mathbf{P}_{s'_T+1}$.
Since $\mathbf{P}_1,\dots,\mathbf{P}_{s'_T+1}\in S^{\mathbf{d}-\mathbf{1}}$ are general, we can reverse the role of $\mathbf{P}_{s'_T+1}$ and $\mathbf{P}_{k}$ for any $1\leq k\leq s'_T$, and we obtain that the intersection 
$$S^1_x\cdot  \mathbf{P} \cap \left(T+ \sum_{i=1}^{k-1}  S^1_x \cdot \mathbf{P}_i + \sum_{i=k+1}^{s'_T+1}  S^1_x \cdot \mathbf{P}_i  \right)$$ 
is the same as \eqref{eq:n-1 bis}.
Thus the line $\ell_T$ does not depend on the choice of $\mathbf{P}_1,\dots,\mathbf{P}_{s'_T+1}$, as claimed in (ii).

\smallskip
In order to prove \eqref{eq:ideale^d}, it is enough to show that for a general $\mathbf{Q}\in S^{\mathbf{d}-\mathbf{1}}$, we have 
\begin{equation}\label{eq:inclusion}
S^1_\ell\cdot \mathbf{Q}\subseteq T + \sum_{i=1}^{s'_T}  S^1_x \cdot \mathbf{P}_i.
\end{equation} 
We recall that if $V$ and $W$ are vector spaces, and $Z:=\left\{\left.\phi\in \operatorname{Hom}(V,W)\right| \dim(\operatorname{Im}\phi)\leq k \right\}$ (for a certain integer $k\geq 1$), then the tangent space to $Z$ at $\phi$ is 
$T_{Z,\phi}=\left\{\left.\psi\in \operatorname{Hom}(V,W)\right| \psi(\operatorname{Ker}\phi)\subseteq \operatorname{Im}\phi \right\}$. 
We apply this fact to the map 
$$\nu_{\mathbf{P}}\colon S^1_x\longrightarrow \frac{S^{\mathbf{d}}_x}{T+  \sum_{i=i}^{s'_T}  S^1_x \cdot \mathbf{P}_i} \quad \text{such that} \quad L\longmapsto L\cdot \mathbf{P} \mod  T+  \sum_{i=1}^{s'_T}  S^1_x \cdot \mathbf{P}_i.$$
Letting $\mathbf{Q}\in S^{\mathbf{d}-\mathbf{1}}$ vary, the maps $\nu_{\mathbf{Q}}$ describe a linear subspace of $\operatorname{Hom}(V,W)$. 
Hence $\nu_{\mathbf{Q}}\in T_{Z,\nu_{\mathbf{P}}}$, and we deduce that $\nu_{\mathbf{Q}}(\operatorname{Ker}\nu_{\mathbf{P}})\subseteq \operatorname{Im} \nu_{\mathbf{P}}$. 
Moreover, as $\operatorname{Ker}\nu_{\mathbf{P}}= S^1_{\ell}$ and $S^1_{x}= S^1_{\ell} + \langle L_n\cdot \mathbf{P} \rangle$, we conclude that
$$
S^1_{\ell}\cdot \mathbf{Q}\subseteq S^1_{x}\cdot \mathbf{P} + T +  \sum_{i=1}^{s'_T}  S^1_x\cdot  \mathbf{P}_i= \langle L_n\cdot \mathbf{P} \rangle + T +  \sum_{i=1}^{s'_T}  S^1_x\cdot  \mathbf{P}_i.
$$

Aiming for a contradiction, we suppose that for a general $\mathbf{Q}\in S^{\mathbf{d}-\mathbf{1}}$, the space $S^1_{\ell}\cdot \mathbf{Q}$ is not entirely contained in $T +  \sum_{i=1}^{s'_T}  S^1_x\cdot  \mathbf{P}_i$.
Therefore, $L_n\cdot \mathbf{P}\in S^1_{\ell}\cdot \mathbf{Q} + T +  \sum_{i=1}^{s'_T}  S^1_x\cdot  \mathbf{P}_i$.
Setting $\mathbf{P}_{s'_T+1}=\mathbf{Q}$, we then obtain
$$
S^1_{x}\cdot \mathbf{P}\subseteq T+ \sum_{i=1}^{s'_T+1}  S^1_x \cdot \mathbf{P}_i , 
$$
which contradicts \eqref{eq:n-1 bis}.
Thus, given a general $\mathbf{Q}\in S^{\mathbf{d}-\mathbf{1}}$, we have that $S^1_{\ell}\cdot \mathbf{Q}\subseteq T +  \sum_{i=1}^{s'_T}  S^1_x\cdot  \mathbf{P}_i$, which is \eqref{eq:inclusion}.
Hence we have shown inclusion \eqref{eq:ideale^d}.  We point out that the line $\ell_T$ is uniquely determined by the subspace $T\subset S_x^{\mathbf{d}}$, as $\ell_T= V\left(L_1,\dots, L_{n-1}\right)$ does not depend on the choice of the basis $\{L_1,\dots,L_n\}\subset S_x^{1}$.
Indeed, the intersection in \eqref{eq:n-1} coincides with $S^1_{\ell}\cdot \mathbf{P}=\langle L_1\cdot \mathbf{P},\dots, L_{n-1}\cdot \mathbf{P}\rangle$, and any other basis $\{L'_1,\dots,L'_n\}\subset S_x^{1}$ such that $S^1_{\ell}\cdot \mathbf{P}=\langle L'_1\cdot \mathbf{P},\dots, L'_{n-1}\cdot \mathbf{P}\rangle$, must satisfy $L'_1,\dots, L'_{n-1}\in S^1_{\ell}$, so that $\ell_T=V\left(L'_1,\dots, L'_{n-1}\right)$.
Item (i) is then fully proved.

\smallskip
So, it remains to prove (iii). 
Let $W:=T+S_{\ell}^{\mathbf{d}}$ and let $w:=\operatorname{codim}W$ be its codimension in $S_{x}^{\mathbf{d}}$.
It follows from \eqref{eq:ideale^d} that $W\subseteq T+ \sum_{i=1}^{s'_T}  S^1_x \cdot \mathbf{P}_i$, and since $s'_T<s_T$, we have $w\geq 1$.
As in \eqref{eq:mu}, we consider the linear map 
\begin{equation*}
\mu_{w,W}\colon \left(S^1_x\right)^{\oplus w}\longrightarrow \frac{S_x^{\mathbf{d}}}{W} \quad \text{such that} \quad \left(L_1,\dots,L_w\right) \longmapsto L_1 \cdot \mathbf{P}_1+\dots + L_w\cdot \mathbf{P}_w \mod  W.
\end{equation*}
Since $w=\codim W$, the map $\mu_{w,W}$ is surjective (cf. Remark \ref{rem:monotonia}(i)), that is $W+\sum_{i=1}^{w}  S^1_x\cdot  \mathbf{P}_i=S_x^{\mathbf{d}}$.
We claim that $w=s_T$.
Indeed, if $w\leq s_T-1$, then \eqref{eq:ideale^d} would give
$$
S_x^{\mathbf{d}}=W+\sum_{i=1}^{w}  S^1_x\cdot  \mathbf{P}_i\subseteq T+ \sum_{i=1}^{s_T-1}  S^1_x \cdot \mathbf{P}_i, 
$$
so that the map $\mu_{s_T-1,T}$ is surjective, a contradiction. Hence 
\begin{equation}\label{eq:w=s_T}
    \dim\big(S_{x}^{\mathbf{d}}\big/ W\big)=w=s_T.
\end{equation}

We note further that
$$\dim\frac{S_{x}^{\mathbf{d}}}{S_{\ell}^{\mathbf{d}}}=\sum_{i=1}^c \dim \frac{S_{x}^{d_i}}{S_{\ell}^{d_i}}=\sum_{i=1}^c d_i=d.$$
Therefore, using $\dim\left(T+ S_{\ell}^{\mathbf{d}}\right)+\dim\left(T\cap S_{\ell}^{\mathbf{d}}\right)=\dim T+\dim S_{\ell}^{\mathbf{d}}$, we obtain
$$\dim\frac{T}{T\cap S_{\ell}^{\mathbf{d}}}=\dim\frac{T+ S_{\ell}^{\mathbf{d}}}{S_{\ell}^{\mathbf{d}}}=\dim\frac{W}{S_{\ell}^{\mathbf{d}}}=\dim\frac{S_{x}^{\mathbf{d}}}{S_{\ell}^{\mathbf{d}}}-\dim\frac{S_{x}^{\mathbf{d}}}{W}=d-s_T,
$$
where in the last equality we used (\ref{eq:w=s_T})
So  the first chain of equalities of (iii) is proved.
In order to deduce the second one, notice that \eqref{eq:T} gives
$$
\dim\frac{S_{\ell}^{\mathbf{d}}}{T\cap S_{\ell}^{\mathbf{d}}}=\dim\frac{T+ S_{\ell}^{\mathbf{d}}}{T}=\dim W - \dim T=t-w=t-s_T=\gamma(s'_T)-s'_T,
$$
as claimed.
\end{proof}
\end{lemma}

More generally, let $Y$ be a smooth variety  endowed with a morphism $f\colon Y\longrightarrow \mathbb{P}^n$, and let us consider a subbundle $\mathcal{T}\subset f^*({M}_{\mathbf{d}})$ of corank $=t$, where ${M}_{\mathbf{d}}$ is the vector bundle defined in (\ref{eq:M_d}). 
Given general elements $\mathbf{P}_1,\dots,\mathbf{P}_s\in S^{\mathbf{d}-\mathbf{1}}$, we introduce the map
\begin{equation}\label{eq:mu2} 
\mu_{s,\mathcal{T}}\colon \left(f^*({M}_{\mathbf{d}})\right)^{\oplus s}\longrightarrow \frac{f^*({M}_{\mathbf{d}})}{\mathcal{T}},
\end{equation}
defined as in \eqref{eq:mu}. 

The following result is the analog of \cite[Lemma 3.1]{CR04}.

\begin{lemma}\label{lem:3.1}
Let $s_\mathcal{T}$ be the least integer such that the map \eqref{eq:mu2} is generically surjective.
Then
$$H^0\left( \det \left(\frac{f^*({M}_{\mathbf{d}})}{\mathcal{T}}\right)\otimes f^*\mathcal O_{\mathbb P^n}(s) \right)\not= 0.$$
\begin{proof}
To lighten the notation, for any integer $m$ we will write  $\mathcal O(m)$ instead of $\mathcal O_{\mathbb P^n}(m)$. 
We consider the map 
\begin{equation}\label{eq:3.2}
    \oplus_{i=1}^s f^*({M}_1)\longrightarrow \frac{f^*({M}_{\mathbf{d}})}{\mathcal{T}},\ (L_1,\ldots, L_s)\mapsto L_1\cdot \mathbf{P}_1+\dots +L_s\cdot \mathbf{P}_s\ \mod{\mathcal T}.
\end{equation}
    Taking its $t$-th exterior power and tensoring it with $\mathcal O(s)$ we get a generically surjective map
    $$
    \bigoplus_{i_1+\dots+i_s} \left( \wedge^{i_1}f^*({M}_1)\otimes \ldots \otimes \wedge^{i_s}f^*({M}_1) \right)\otimes \mathcal O(s) =\left( \bigwedge^t ( \oplus_{i=1}^s f^*({M}_1)\right) \otimes \mathcal O(s)\longrightarrow \left(\bigwedge^t \frac{f^*({M}_{\mathbf{d}})}{\mathcal{T}}\right)\otimes \mathcal O(s).
    $$
    As we have 
    $$
    \bigoplus_{i_1+\dots+i_s} \left( \wedge^{i_1}f^*({M}_1)\otimes \ldots \otimes \wedge^{i_s}f^*({M}_1) \right)\otimes \mathcal O(s) =
    \bigoplus_{i_1+\dots+i_s} \left( \wedge^{i_1}( f^*({M}_1)\otimes \mathcal O(1) ) \otimes \ldots \otimes \wedge^{i_s}( f^*({M}_1)\otimes \mathcal O(1) )\right)
    $$
    we obtain a generically surjective map from a globally generated vector bundle onto 
    $$
    \det \left(\frac{f^*({M}_{\mathbf{d}})}{\mathcal{T}}\right)\otimes f^*\mathcal O_{\mathbb P^n}(s) 
    $$
    and we are done. 
\end{proof}
\end{lemma}
Recall that we are in the following framework. 
We consider the universal family $\mathcal X\longrightarrow S^\d$ of complete intersections of type $\d\coloneqq(d_1,\ldots, d_c)$ in $\mathbb P^n$, and a family of $k$-dimensional subvarieties $\mathcal Y\subset \mathcal X$ (we omit the base-change to render the notation less heavy).

\begin{lemma}\label{lem:neq 0}
Let $\mathcal{T}\coloneqq T_{\mathcal{Y}}^{\textrm{vert}}$ and let $s_\mathcal{T}$ be the least integer such that the map \eqref{eq:mu2} is generically surjective. Then
\begin{equation}\label{eq:nonvan}
    H^0\left(\omega_{\widetilde Y_{\mathbf{F}}}(-d+n+1+s_\mathcal{T})\right)\neq 0.
\end{equation}
In particular, if for $a\in\{0,1\}$ 
the restriction map
\begin{equation}\label{eq:vanish}
  H^0\left({\Omega^{N+k}_{\mathcal X}}_{|X_{\bf F}}(-a)\right)\longrightarrow H^0\left({\Omega^{N+k}_{\tilde {\mathcal Y}}}_{|\tilde Y_{\bf F}}(-a)\right)
\end{equation}
vanishes, 
then
\begin{equation}\label{eq:ineq-s_T}
   d-a<n+1+s_\mathcal{T}. 
\end{equation}
\end{lemma}
\begin{proof}
    From Lemma \ref{lem:3.1} and equations (\ref{eq:normal}) and (\ref{eq:M-Tvert}) restricted to a general  fiber over $\bf{F}\in S^{\bf d}$, we deduce that 
    $$
    h^0\left(\det (N_{\widetilde {\mathcal Y}/{\mathcal X}}\right)_{|\widetilde Y_{\bf F}}\otimes f^*\mathcal O_{\mathbb P^n}(s))>0. 
    $$
    Since by adjunction 
    $$\det \left(N_{\widetilde {\mathcal Y}/{\mathcal X}}\right)_{|\widetilde Y_{\bf F}}= \omega_{\tilde Y_{\bf F}}\otimes j^*\omega^{-1}_{X_{\bf F}}=\omega_{\widetilde Y_{\bf F}} (-d+n+1),
    $$
    by twisting with $s_{\mathcal T}$ we obtain the non-vanishing (\ref{eq:nonvan}). Inequality (\ref{eq:ineq-s_T}) follows immediately from (\ref{eq:nonvan}).
\end{proof}

\begin{corollary}\label{cor:s-s'}
Let $\mathcal{T}\coloneqq T_{\mathcal{Y}}^{\textrm{vert}}$ and let $s_\mathcal{T}$ (respectively $s'_\mathcal{T}$) be the least integer such that such that the map \eqref{eq:mu2} is generically surjective (resp. such that generically $\rk (\mu_{s+1, \mathcal{T}})-\rk (\mu_{s, \mathcal{T}})\leq 1$).
If \eqref{eq:bound lineare} holds and the restriction map \eqref{eq:vanish} vanishes, 
then 
$$s_\mathcal{T}-s_\mathcal{T}'\geq 2 \quad \text{and}\quad s_\mathcal{T}> \left\lfloor\frac{n-k-c+1}{2}\right\rfloor +a.$$
\begin{proof}
We recall that, by (\ref{eq:normal}) and (\ref{eq:M-Tvert}), the corank of $\mathcal{T}$ in $g^*\left(\mathcal{M}^{\mathbf{d}}_{\mathbb{P}^n}\right)$ is $n-k-c$.
If $s_\mathcal{T}-s_\mathcal{T}'=0$, it follows from \eqref{eq:s'=s} that $s_\mathcal{T}\leq \frac{n-k-c}{2}$.
If instead $s_\mathcal{T}-s'_\mathcal{T}=1$, then \eqref{eq:t} yields $n-k-c\geq s_\mathcal{T}+s'_\mathcal{T}=2s_\mathcal{T}-1$.
Therefore, in both cases we have that $s_\mathcal{T}\leq \left\lfloor\frac{n-k-c+1}{2}\right\rfloor$, and hence $H^0\left(\omega_{Y_{\mathbf{F}}}\left(-d+n+1+\left\lfloor\frac{n-k-c+1}{2}\right\rfloor\right)\right)\neq 0$ by Lemma \ref{lem:neq 0}. 
Since the restriction map (\ref{eq:vanish}) vanishes, 
we also deduce that $d -a<n+1+\left\lfloor\frac{n-k-c+1}{2}\right\rfloor$, which contradicts \eqref{eq:bound lineare}.
Thus we conclude that $s_\mathcal{T}-s'_\mathcal{T}\geq 2$ and $s_\mathcal{T}> \left\lfloor\frac{n-k-c+1}{2}\right\rfloor+a$.
\end{proof}
\end{corollary}
\begin{remark}\label{rmk:mumford}
    The restriction map \eqref{eq:vanish} obviously vanishes if  
$H^0\left(\omega_{\widetilde Y_{\mathbf{F}}}(-a) \right) = 0$. Another important case, which will be used in the proof of Theorem \ref{thm:CLS+}, is when $a=0$ and $Y_\F$ is a $k$-dimensional subvariety whose points are all rationally equivalent in $X$ (where by rational equivalence we mean $\mathbb Q$-rational equivalence). Indeed the variational (and higher dimensional) version of Mumford's theorem on $0$-cycles yields in this case the vanishing of the map (\ref{eq:vanish}) (see \cite[Proposition 22.24]{Vbook} for the absolute case, \cite[Proposition 1.8]{V94} for the case of relative dimension 2, and \cite[Lemma 1]{V2} and \cite[Section 2.2]{P03} for the same remark in the general case). 
\end{remark}

\section{Integrability of the vertical contact distribution and proof of Theorem \ref{thm:CLS+}}\label{s:int}

Using the notation introduced in Section \ref{s: preliminaries}, we consider the universal family $\mathcal X\longrightarrow S^\d$  of complete intersections of type $\d\coloneqq(d_1,\ldots, d_c)$ in $\mathbb P^n$ and a family of $k$-dimensional subvarieties $\mathcal Y\subset \mathcal X$ (we omit the base-change to render the notation less heavy).
We assume that \eqref{eq:bound lineare} holds and that the restriction map (\ref{eq:vanish}) vanishes.

In the light of Corollary \ref{cor:s-s'} and Lemma \ref{lem:retta}, we can consider the rational map 
\begin{equation}\label{eq:phi}
\phi\colon \mathcal{Y}\dashrightarrow \mathbb{G}(1,n)
\end{equation}
which sends a general point $(y,\mathbf{F})\in \mathcal{Y}$ to the line $\ell_{(y,\mathbf{F})}\coloneqq\ell_T$, with $T=T_{\mathcal{Y}, (y,\mathbf{F})}^{\textrm{vert}}\cap S_y^{\mathbf{d}}$. 
Hence we have a rational map 
\begin{equation}\label{eq:psi}
    \psi\colon \mathcal{Y}  \dashrightarrow \Delta\coloneqq \left\{(x,[\ell],\mathbf{G})\left|x\in \ell\cap X_{\mathbf{G}}\right.\right\}\subset \mathbb{P}^n\times \mathbb{G}(1,n)\times S^{\mathbf{d}}\\
\end{equation}
sending a general $(y,\mathbf{F})$ to the triple $\left(y,\left[\ell_{(y,\mathbf{F})}\right],\mathbf{F}\right)$.

We note that $\mathcal{T}\coloneqq T_{\mathcal{Y}}^{\textrm{vert}}$ can be thought as a subbundle of $S^{\mathbf{d}}\otimes \mathcal{O}_{\mathcal{Y}}$.
Moreover we may define another subbundle $\mathcal{L}$ of $S^{\mathbf{d}}\otimes \mathcal{O}_{\mathcal{Y}}$, whose fibre at $(y,F)\in \mathcal{Y}$ consists of $c$-tuples $\mathbf{P}\in S^{\mathbf{d}}$ vanishing on the line $\ell_{(y,\mathbf{F})}$, i.e.
\begin{equation}\label{eq:mathcal L}
    \mathcal{L}_{(y,\mathbf{F})}=S^{\mathbf{d}}_{\ell_{(y,\mathbf{F})}}=\bigoplus_{i=1}^c H^0\left(\mathbb{P}^n, \mathcal{I}_{\ell_{(y,\mathbf{F})}}(d_i)\right).
\end{equation}

\smallskip
A multi-polynomial $\F=(F_1,\ldots,F_c)\in S^\d$ can then be seen also as a tangent vector. Therefore we introduce the same notation as \cite[Section 6]{CR04}, adapting it to the case of complete intersections. Let us fix a point $y\in \mathbb{P}^n$ and let us assume that $y$ has coordinates $[1,0,\dots,0]$. More explicitly, for each $i=1,\ldots, c$ we can write 
$$
F_i (X_0,\ldots, X_n)=\sum_{|A|\leq d_i} a_{i,A} X_0^{d_i-|A|}X^A
$$
where $A=(A(1),\ldots, A(n))$ is a multi-index of length $|A|:=A(1)+\ldots+ A(n)\leq d_i$ and $X^A$ denotes the polynomial 
$$ 
X^A=X_1^{A(1)}\cdots X_n^{A(n)}.
$$
If we see $\{X_0^{d_i-|A|}X^A\}$ as a basis of $S^{d_i}$ the corresponding coefficients $a_{i,A}$ yield naturally tangent vectors seen as differentiation operators denoted by $\partial^i_A = \partial/\partial{a_{i,A}}$.
In what follows, for a multi-index $A=(A(1),\ldots, A(n))$ and an integer $1\leq r\leq n$ (or two integers $1\leq r,s\leq n$) we will also use the following notation 
$Ar$ (respectively $Ars$) meaning
$$
Ar=(A(1),\ldots,A(r)+1,\ldots, A(n)),
$$
respectively 
$$
Ars=(A(1),\ldots,A(r)+1,\ldots, A(s)+1,\ldots,A(n)).
$$
Therefore, the corresponding polynomials are
$$
X^{Ar}=X^A\cdot X_r
\quad\text{and} \quad
X^{Ars}=X^A\cdot X_r\cdot X_s
$$

Let
$$\mathcal{Y}_0\coloneqq \pi^{-1}(y)\subset \mathcal{Y}$$ 
be the fibre over $y$ of the projection $\mathcal{Y}\stackrel{\pi}{\longrightarrow} \mathbb{P}^n$.
Let us focus on an analytic neighborhood $U$ of a general point $(y,\mathbf{F})\in \mathcal{Y}_0$.
Of course $U$ maps isomorphically to $S'$, which is \'etale over $S\coloneqq S^{\mathbf{d}}\smallsetminus \{0\}$, and hence $U$ is isomorphic to an open subset of $S$.\\
 Since $S$ is open in $S^{\mathbf{d}}$ we note that 
\begin{equation}\label{eq:tangent analytical}
T_S\cong S^{\mathbf{d}}\otimes \mathcal{O}_S \quad \text{and} \quad T_S\otimes \mathcal{O}_{\mathcal{Y}_0}\cong S^{\mathbf{d}}\otimes \mathcal{O}_{\mathcal{Y}_0}.
\end{equation}

Let us assume that the line $\ell_0\coloneqq \ell_{(y,\mathbf{F})}$ has equations $X_2=\dots=X_n=0$, and for $(y,\mathbf{G})\in U$, the corresponding line has equations $\ell_{(y,\mathbf{G})}=\phi(y,\mathbf{G})$
\begin{equation}\label{eq:b_r}
    X_r=b_r(y,\mathbf{G})X_1\quad \text{for }r=2,\dots,n,
\end{equation}
where the $b_r$'s are holomorphic functions on $U$.
In particular, $b_2(y,\mathbf{F})=\dots=b_n(y,\mathbf{F})=0$, by the choice of the coordinates above.

Under the identification \eqref{eq:tangent analytical}, the generators of the bundle $\mathcal{L}$ are the vector fields
$$
\partial^i_{Ar}-b_r\partial^i_{A1}\in T_S\otimes\mathcal{O}_U\quad \text{with } |A|\leq d_i-1,\quad r=2,\dots,n,\quad i=1,\dots,c,
$$
each corresponding to the polynomial $X_0^{d_i-1-|A|}X^A\left(X_r-b_rX_1\right)$ of degree $d_i$ on the $i$-th component of $S^\mathbf{d}=\bigoplus_{i=1}^c S^{d_i}$.

\subsection{Integrability of the vertical contact distribution}\label{ss:integrability}

We consider the restriction to $\mathcal{Y}_0$ of the bundle $\mathcal{L}$ defined in \eqref{eq:mathcal L}, and we set
$$\mathcal{T}'\coloneqq T_{\mathcal{Y}_0}\cap \mathcal{L},$$
where we note that the restriction of $T_{\mathcal{Y}}^{\textrm{vert}}$ to $\mathcal{Y}_0$ is $T_{\mathcal{Y}_0}$.
By arguing as in \cite[Lemmas 6.1 and 6.2]{CR04} we prove the following.

\begin{proposition}\label{prop:b_r}
    For any $r\geq 2$ and for any section $\sigma$ of $\mathcal{T}'$, we have that $\sigma(b_r)=0$. In particular, the distribution $\mathcal{T}'$ is integrable and the leaf at a point $(y,\mathbf G)$ is given by the affine space $(y,\mathbf G)+ \mathcal T'_{(y,\mathbf G)}$. 

\begin{proof}
    It sufficies to prove the assertion at the general point $(y, \mathbf{F})$.
    To this aim, we define the subbundle $\mathcal{L}_2$ of $S^{\mathbf{d}}\otimes \mathcal{O}_{U}$, whose fibre at $(y,\mathbf{G})\in U$ consists of $c$-tuples $\mathbf{P}\in S^{\mathbf{d}}$ vanishing with order two on the line $\ell_{(y,\mathbf{G})}$, i.e.
    \begin{equation}\label{eq:mathcal L2}
    \mathcal{L}_{2,(y,\mathbf{G})}=S^{\mathbf{d}}_{2\ell_{(y,\mathbf{G})}}=\bigoplus_{i=1}^c H^0\left(\mathbb{P}^n, \mathcal{I}_{2\ell_{(y,\mathbf{G})}}(d_i)\right).
    \end{equation}
    Then we consider the bundle
    $$\mathcal{T}''\coloneqq \mathcal{T}'\cap \mathcal{L}_2,$$
    and we note that the sections of $\mathcal{T}''$ are generated by the vector fields
    $$
    \partial^i_{Ars}-b_r\partial^i_{A1s}-b_s\partial^i_{A1r}+b_rb_s\partial^i_{A11}\quad \text{with } |A|\leq d_i-2,\quad 2\leq r,s\leq n,\quad i=1,\dots,c,
    $$
    corresponding to the polynomials $X_0^{d_i-2-|A|}X^A\left(X_r-b_rX_1\right)\left(X_s-b_sX_1\right)$ of degree $d_i$.
    It is easy to check that given $i,j=1,\dots,c$, for any multi-indices $A,B$ with $|A|\leq d_i-2$ and $|B|\leq d_j-2$, and for any $2\leq r,s,p,q\leq n$, the bracket
    $$
    \left[\partial^i_{Ars}-b_r\partial^i_{A1s}-b_s\partial^i_{A1r}+b_rb_s\partial^i_{A11}\, , \, \partial^j_{Bpq}-b_p\partial^j_{B1q}-b_q\partial^j_{B1p}+b_pb_q\partial^j_{B11}\right]
    $$
    is a combination of elements of the form $\partial^k_{C1m}-b_m\partial^k_{C11}$, with $(k,C)\in\left\{(i,A),(j,B)\right\}$ and $m\in\left\{r,s,p,q\right\}$.
    Therefore $\left[\mathcal{T}'',\mathcal{T}''\right]\subset \mathcal{T}'$, and we have a well-defined bracket pairing
    \begin{equation}\label{eq:bracket1}
    \mathcal{T}''\times \frac{\mathcal{T}'}{\mathcal{T}''}\longrightarrow \frac{\mathcal{T}}{\mathcal{T}'}.    
    \end{equation}
    
    Now, want to prove that this pairing vanishes, and we start by describing the bundles $\mathcal{T}'\big/ \mathcal{T}''$ and $\mathcal{T}\big/ \mathcal{T}'$.  
    We note that, using the identification $\partial^i_{Ars}-b_s\partial^i_{A1r}=b_r\partial^i_{A1s}-b_rb_s\partial^i_{A11}$ recursively, one can easily see that the sections of $\mathcal{T}'\big/ \mathcal{T}''$ have the form  
    \begin{equation}\label{eq:tau'}
    \tau'=\sum_{i=1}^c\tau'_i=\sum_{i=1}^c\left(\sum_{r=2}^n\sum_{\alpha=0}^{d_i-1}e_{r\alpha}^i(\tau')\left(\partial^i_{r1^{\alpha}}-b_r\partial^i_{1^{\alpha+1}}\right)\right), 
    \end{equation}
    where $1^{\alpha}$ means that the index 1 is repeated $\alpha$ times.
    Then we can associate to each section $\tau'$ a $(n-1)\times d$ matrix given by    
    $$
    E(\tau')=\big(E^1\big|\cdots \big|E^c\big)\quad \text{with }E^i\coloneqq \left(e^i_{r\alpha}(\tau')\right)_{\substack{r=2,\dots,c \\ \alpha=0,\dots , d_i}}. 
    $$
    Up to shrinking $U$, we can think of $\mathcal{T}'\big/ \mathcal{T}''$ as a subbundle of the bundle $M_{(n-1)\times d}(\mathcal{O}_U)$ of $\mathcal{O}_U$-valued matrices of dimension $(n-1)\times d$.
    Furthermore, since $\mathcal{T}''\subset \mathcal{L}_2$ and $\rk\left(\mathcal{L}\big/ \mathcal{L}_2\right)=\dim \left( S^{\mathbf{d}}_{\ell_0}\big/ S^{\mathbf{d}}_{2\ell_0}\right)=(n-1)d$ equals the rank of $M_{(n-1)\times d}(\mathcal{O}_U)$, we deduce that the corank of $\mathcal{T}'\big/ \mathcal{T}''$ in $M_{(n-1)\times d}(\mathcal{O}_U)$ is
    \begin{equation}\label{eq:corank1}
    \rk M_{(n-1)\times d}(\mathcal{O}_U)- \rk \frac{\mathcal{T}'}{\mathcal{T}''}\leq \rk \frac{\mathcal{L}}{\mathcal{T}'}=\dim \frac{S^{\mathbf{d}}_{\ell_0}}{T^{\mathrm{vert}}_{\mathcal{Y},(y,\mathbf{F})}\cap S^{\mathbf{d}}_{\ell_0}}.
    \end{equation}
    Thanks to Corollary \ref{cor:s-s'}, we can apply Lemma \ref{lem:retta}(iii), and by the first equality in \eqref{eq:t} we conclude that the corank of $\mathcal{T}'\big/ \mathcal{T}''$ is at most $n-k-c-s_{\mathcal{T}}$.

    On the other hand, using the identification $\partial^i_{Ar}=b_r\partial^i_{A1}$ recursively, we obtain that the sections of $\mathcal{T}\big/ \mathcal{T}'$ have the form  
    $$
    \tau=\sum_{i=1}^c\left(\sum_{\beta=1}^{d_i}e_{\beta}^i(\tau)\partial^i_{1^{\beta}}\right), 
    $$
    which can be associated to a vector $\left(e^1_1(\tau),\ldots,e^1_{d_1}(\tau),\ldots,e^c_1(\tau),\ldots,e^c_{d_c}(\tau)\right)$.
    Hence we may view $\mathcal{T}\big/ \mathcal{T}'$ as a subbundle of $\mathcal{O}_U^{\oplus d}$, and by Lemma \ref{lem:retta}(iii), its corank is
    \begin{equation}\label{eq:corank2}
    \rk \mathcal{O}_U^{\oplus d}- \rk \frac{\mathcal{T}}{\mathcal{T}'}=d-\dim \frac{T^{\mathrm{vert}}_{\mathcal{Y},(y,\mathbf{F})}}{T^{\mathrm{vert}}_{\mathcal{Y},(y,\mathbf{F})}\cap S^{\mathbf{d}}_{\ell_0}}=s_{\mathcal{T}}.
    \end{equation}

    We point out that given $i,j=1,\dots,c$, for any multi-index $B$ with $|B|\leq d_j-2$, and for any $2\leq p,q,r\leq n$, we have that the bracket   
    $$
     \Big[\partial^j_{Bpq}-b_p\partial^j_{B1q}-b_q\partial^j_{B1p}+b_pb_q\partial^j_{B11} \,, \, \partial^i_{r1^{\alpha}}-b_r\partial^i_{1^{\alpha+1}}\Big]
     $$
     equals $-\left(\partial^j_{Bpq}(b_r)-b_p\partial^j_{B1q}(b_r)-b_q\partial^j_{B1p}(b_r)+b_pb_q\partial^j_{B11}(b_r)\right)\partial^i_{1^{\alpha+1}}$ modulo $\mathcal{T}'$.
     Therefore, considering a section $\tau''$ of $\mathcal{T}''$ given by
     \begin{equation*}\label{eq:tau''}
    \tau''=\sum_{j=1}^c\tau''_j=\sum_{j=1}^c \left(\sum_{p,q=2}^n\sum_{|B|\leq d_j-2}e_{Bpq}^j(\tau'')\left(\partial^j_{Bpq}-b_p\partial^j_{B1q}-b_q\partial^j_{B1p}+b_pb_q\partial^j_{B11}\right)\right), 
    \end{equation*}
     and a section $\tau'$ as in $\eqref{eq:tau'}$, the pairing \eqref{eq:bracket1} is
     $$
     \left[\tau'',\tau'\right]=-\sum_{i,j} \sum_{r,\alpha}\tau''_j(b_r)e^i_{r\alpha}(\tau')\partial^i_{1^{\alpha+1}}=-\sum_i \sum_{r,\alpha}\left(\sum_j\tau''_j(b_r)\right)e^i_{r\alpha}(\tau')\partial^i_{1^{\alpha+1}}.
     $$

    \begin{claim}\label{clm:T''}
    For any $2\leq r\leq n$ and for any section $\tau''$ of $\mathcal{T}''$, we have that $\tau''(b_r)=\sum\tau''_j(b_r)=0$.
    \begin{proof}[Proof of Claim \ref{clm:T''}]
        Suppose by contradiction that $\tau''(b_r)\neq 0$ for some $\tau''$ and $r\in \{2,\dots,n\}$, and let us consider the map $\left[\tau'',\,\cdot\,\right]\colon \mathcal{T}'\big/ \mathcal{T}''\longrightarrow \mathcal{T}\big/ \mathcal{T}'$.
        Viewing $\mathcal{T}'\big/ \mathcal{T}''$ and $\mathcal{T}\big/ \mathcal{T}'$ as subbundles of $M_{(n-1)\times d}(\mathcal{O}_U)$ and $\mathcal{O}_U^{\oplus d}$ respectively, $\left[\tau'',\,\cdot\,\right]$ is the restriction of the map 
        $$\Phi\colon M_{(n-1)\times d}(\mathcal{O}_U)\longrightarrow \mathcal{O}_U^{\oplus d}$$
        given by multiplication by the vector $\varphi(\tau'')\coloneqq\left(\tau''(b_2),\dots,\tau''(b_n)\right)$.
        Since $\tau''(b_r)\neq 0$, the map $\Phi$ is surjective, and hence the corank of $\Phi(\mathcal{T}'\big/ \mathcal{T}'')$ in $\mathcal{O}_U^{\oplus d}$ cannot be larger than the corank of $\mathcal{T}'\big/ \mathcal{T}''$ in $M_{(n-1)\times d}(\mathcal{O}_U)$.
        Therefore, using \eqref{eq:corank1} and \eqref{eq:corank2}, we obtain
        \begin{equation}\label{eq:corank3}
        s_{\mathcal{T}} = \rk \mathcal{O}_U^{\oplus d}- \rk \frac{\mathcal{T}}{\mathcal{T}'}\leq \rk \mathcal{O}_U^{\oplus d}-\rk\Phi(\mathcal{T}'\big/ \mathcal{T}'')\leq 
        \rk M_{(n-1)\times d}(\mathcal{O}_U)- \rk\frac{\mathcal{T}'}{\mathcal{T}''}\leq n-k-c-s_{\mathcal{T}},
        \end{equation}
        which contradicts Corollary \ref{cor:s-s'}.
    \end{proof}
    \end{claim}

    In the light of the claim, we deduce that $\left[\tau'',\tau'\right]=0$, i.e. the pairing \eqref{eq:bracket1} vanishes.
    Thus we have a well-defined bracket pairing 
    \begin{equation}\label{eq:bracket2}
    \frac{\mathcal{T}'}{\mathcal{T}''}\times \frac{\mathcal{T}'}{\mathcal{T}''}\longrightarrow \frac{\mathcal{T}}{\mathcal{T}'}.    
    \end{equation}
    Given a section $\sigma'=\sum_{j}\sigma'_j=\sum_{j}\left(\sum_{s}\sum_{\beta}e_{s\beta}^j(\sigma')\left(\partial^j_{s1^{\beta}}-b_s\partial^j_{1^{\beta+1}}\right)\right)$ and a section $\tau'$ as in \eqref{eq:tau''}, the bracket gives
    \begin{equation}\label{eq:bracket3}
    \left[\tau',\sigma'\right]=
    \sum_{i=1}^c\sum_{\alpha=0}^{d_i-1}\sum_{r=2}^n\left( \sum_{j=1}^c\sigma'_j(b_r)\right)e_{r\alpha}^i(\tau')\partial^i_{1^{\alpha+1}} 
    -\sum_{j=1}^c\sum_{\beta=0}^{d_j-1}\sum_{s=2}^n\left( \sum_{i=1}^c\tau'_i(b_s)\right)e_{s\beta}^j(\sigma')\partial^j_{1^{\beta+1}}.
    \end{equation}
    In order to conclude the proof, we need to show that $\sigma'(b_r)=\sum\sigma'_j(b_r)=0$ for any $2\leq r\leq n$.
    
    To this aim, we us consider the map $\varphi\colon \mathcal{T}'\big/ \mathcal{T}''\longrightarrow \mathcal{O}_U^{\oplus n-1}$ such that $\varphi(\tau')\coloneqq\left(\tau'(b_2),\dots,\tau'(b_n)\right)$.
    Thinking of $\mathcal{T}'\big/ \mathcal{T}''$ and $\mathcal{T}\big/ \mathcal{T}'$ as subbundles of $M_{(n-1)\times d}(\mathcal{O}_U)$ and $\mathcal{O}_U^{\oplus d}$ respectively, we point out that
    $$
    \left[\tau',\sigma'\right]=\varphi(\sigma')\cdot\tau'-\varphi(\tau')\cdot\sigma'.
    $$
    Then we recall the following linear algebra result from \cite{CR04}.
    \begin{claim}\label{clm:linear algebra}
        Let $V,W$ be finite-dimensional vector spaces and let $H<\Hom(V,W)$ be a subspace of codimension $\epsilon$.
        Let $\varphi\colon H\longrightarrow V$ be a nonzero linear map and define a pairing
        $$
        I(\varphi)\colon \bigwedge^2 H \longrightarrow W \quad \text{such that} \quad A\wedge B\longmapsto A\varphi(B)-B\varphi(A).
        $$
        Then the image of $I(\beta)$ is of codimension at most $\epsilon+1$ in $W$.
    \begin{proof}[Proof of Claim \ref{clm:linear algebra}]
        See \cite[Sublemma 6.1.1]{CR04}.
    \end{proof}
    \end{claim}
    Let $H$, $V$ e $W$ be the fibers at $(y,\mathbf{F})$ of the bundles $\mathcal{T}'\big/ \mathcal{T}''$, $\mathcal{O}_U^{\oplus n-1}$ and $\mathcal{O}_U^{\oplus d}$, respectively.
    In particular, the codimension of $H$ in $\Hom(V,W)$ is $\epsilon\leq n-k-c-s_{\mathcal{T}}$.
    
    Suppose by contradiction that $\sigma'(b_r)=\sum\sigma'_j(b_r)\neq 0$ for some $2\leq r\leq n$.  
    As $(y,\mathbf{F})\in \mathcal{Y}_0$ is a general point, the restriction of $\varphi\colon \mathcal{T}'\big/ \mathcal{T}''\longrightarrow \mathcal{O}_U^{\oplus n-1}$ to the fibers at $(y,\mathbf{F})$ is nonzero, and the restriction of the bracket at \eqref{eq:bracket2} is the composition of $I(\varphi)$ with the natural map $H\times H\longrightarrow \bigwedge^2 H$. 
    Therefore, Claim \ref{clm:linear algebra} ensures that the image of the pairing \eqref{eq:bracket2} has corank at most $\epsilon+1\leq n-k-c-s_{\mathcal{T}} + 1$ in $\mathcal{O}_U^{\oplus d}$.
    Finally, using \eqref{eq:corank2} and arguing as in \eqref{eq:corank3}, we obtain that $s_{\mathcal{T}}\leq n-k-c-s_{\mathcal{T}} + 1$, which contradicts Corollary \ref{cor:s-s'}.
    In conclusion, $\sigma'(b_r)=0$ for any $2\leq r\leq n$, and the first assertion follows. 
    
    It follows that the bracket \eqref{eq:bracket3} vanishes, i.e. $\left[\mathcal{T}',\mathcal{T}'\right]\subset \mathcal{T}'$, and hence the distribution $\mathcal{T}'$ is integrable by Frobenius theorem. 

    Finally, we note from (\ref{eq:b_r}) that there is a natural identification between the differential of $\phi$ at a point $(y,\mathbf{G})$ and the restriction of $\varphi$ to the fibers at $(y,\mathbf{G})$.
    Therefore, the fact that $\sigma'(b_r)=0$ for any section $\sigma'$ of $\mathcal T'$ means exactly that the line $\ell_{(y,\mathbf{G})}$ does not change along a leaf of the integrable foliation given by $\mathcal T'$. 
\end{proof}
\end{proposition}

\subsection{Special subvarieties are included in contact loci}\label{ss:inclusion}

\begin{proposition}\label{prop:Delta_r}
Assume \eqref{eq:bound lineare} holds.
    Let $\psi\colon \mathcal{Y}  \dashrightarrow \Delta$ be the map in \eqref{eq:psi}.
    Let $(y,\mathbf{F})\in \mathcal{Y}$ be a general point, with $\mathbf{F}=\left(F_1,\dots,F_c\right)$, and let $\widetilde{Y}_\mathbf{F}\longrightarrow Y_\mathbf{F}$ be a desingularization. 
    \begin{itemize}
        \item[\textit{(i)}] If  $a=0$ and the map \eqref{eq:vanish} vanishes, then there exists an integer $r\geq 1$ such that
        $$
    \psi\left(\mathcal{Y}\right)\subset \widetilde{\Delta}_{r}\coloneqq \left\{(x,[\ell],\mathbf{G})\in \Delta\left|\,\ell\cdot V(G_i)\geq rx \text{ if }d_i=r \text{ and } \ell\subset V(G_i)\text{ otherwise}\right.\right\},
    $$
        \item[\textit{(ii)}] If $a=1$ and $h^0(K_{\widetilde Y_\mathbf{F}}(-1))=0$, then there exists an integer $r\geq 1$ such that
        $$
    \psi\left(\mathcal{Y}\right)\subset \widetilde{\Delta}_{r,\mathbf{d}-r}\coloneqq \left\{(x,[\ell],\mathbf{G})\in \Delta\left|\,\exists x'\in \ell \text{ such that }\ell\cdot V(G_i)\geq rx+(d_i-r)x'\ \forall i=1,\dots,c\right.\right\}.
    $$
    \end{itemize}
\begin{proof}
Using notation as above, let $y=[1,0,\dots,0]\in \mathbb{P}^n$ and let $\mathcal{Y}_0
\coloneqq \pi^{-1}(y)\subset \mathcal{Y}$ be the fibre over $y$ of the projection $\mathcal{Y}\stackrel{\pi}{\longrightarrow} \mathbb{P}^n$, that is the multi-polynomials $\mathbf{G}\in S^\d$, for which we have a subvariety $Y_{\mathbf{G}}$ in our family $\mathcal Y$ that passes through the point $y$.
Let $G_y< \GL(n+1)$ be the stabilizer of $y$, and notice that, since $\mathcal{Y}$ is $\GL(n+1)$-invariant, the fiber  $\mathcal{Y}_0$ is invariant under the natural action of $G_y$ on $S^{\mathbf{d}}$.
Furthermore, the restriction $\phi_{|\mathcal{Y}_0}\colon \mathcal{Y}_0 \dashrightarrow \mathbb{G}(1,n)$ of the map defined in \eqref{eq:phi} is $G_y$--equivariant, and hence it is generically surjective on the $(n-1)$-dimensional locus of lines through $y$.

Let us consider the line $\ell\coloneqq \ell_{(y,\mathbf{F})}=\phi(y,\mathbf{F})$, whose equations are assumed to be 
$$
\ell\colon \, X_2=\dots=X_n=0,
$$
and let $\mathcal{Y}_{y,\ell}\coloneqq \phi_{|\mathcal{Y}_0}^{-1}([\ell])\subset \mathcal{Y}_0$ be the fibre of $[\ell]$.
Again, to spell it out explicitly, the fiber $\mathcal{Y}_{y,\ell}$ parametrizes multi-polynomials  $\bf G\in S^\d$, for which we have a subvariety $Y_{\bf G}$ in our family $\mathcal Y$ passing through the point $y$ and for which the associated line is $\ell$.
Since $(y,\mathbf{F})\in \mathcal{Y}$ is a general point, then $\codim_{\mathcal{Y}_0}\mathcal{Y}_{y,\ell}=\dim \phi(\mathcal{Y}_0)=n-1$. 
Moreover, we define the map
$$
\rho\colon \mathcal{Y}_{y,\ell}\longrightarrow \bigoplus_{i=1}^c H^0\left(\ell,\mathcal{O}_{\ell}(d_i)(-y)\right) \quad \text{such that} \quad (y,\mathbf{G})\longmapsto \mathbf{G}_{|\ell}.
$$
It follows from Proposition \ref{prop:b_r} that $\mathcal{T}'=T_{\mathcal{Y}_0}\cap \mathcal{L}$ is the vertical tangent space of the map $\rho$, i.e. its restriction fits in the exact sequence
$$
0\longrightarrow \mathcal{T}'_{|\mathcal{Y}_{y,\ell}}\longrightarrow T_{\mathcal{Y}_{y,\ell}}\stackrel{\rho_*}{\longrightarrow} \bigoplus_{i=1}^c H^0\left(\ell,\mathcal{O}_{\ell}(d_i)(-y)\right)\longrightarrow 0.
$$
This fact and Lemma \ref{lem:retta}(iii) give that
\begin{align*}
\dim \rho(\mathcal{Y}_{y,\ell})& =\dim \mathcal{Y}_{y,\ell} - \rk \mathcal{T}' =\dim  \mathcal{Y}_0 - \codim_{\mathcal{Y}_0}\mathcal{Y}_{y,\ell} - \rk (T_{\mathcal{Y}_0}\cap \mathcal{L})\\
& = \rk \big(T_{\mathcal{Y}_0}/(T_{\mathcal{Y}_0}\cap \mathcal{L})\big)- \codim_{\mathcal{Y}_0}\mathcal{Y}_{y,\ell}= d-s_{\mathcal{T}}-n+1.
\end{align*}
Furthermore, by Lemma \ref{lem:neq 0} and Remark \ref{rmk:mumford}, if $h^0(\widetilde{Y}_\mathbf{F},K_{\widetilde{Y}_\mathbf{F}}(-a))=0$ with $a\in \{0,1\}$, or if all the points of $Y_\F$ are rationally equivalent in $X$ and $a=0$, then $d-s_{\mathcal{T}}-n+1\leq 1+a$.
Therefore, Lemma \ref{lem:jacG1} yields
\begin{equation}\label{eq:im Z}
    \dim \big\langle \,\mathbf{F}_{|\ell}, \left.X_1\frac{\partial \mathbf{F}}{\partial X_0}\right|_{\ell}, \left.X_1\frac{\partial \mathbf{F}}{\partial X_1}\right|_{\ell}\big\rangle \leq \dim \rho(\mathcal{Y}_{y,\ell})\leq 1+a.
\end{equation}

We point out that $\rho$ is equivariant under the action of the stabilizer $G_{y,\ell}< \GL(n+1)$ of the pair $(y,\ell)$.
Hence the image $\rho(\mathcal{Y}_{y,\ell})$ is invariant under the action of the stabilizer $H_y<\GL(\ell)$ of $y$.
We note further that $H_y$ acts doubly-transitively on $\ell\smallsetminus\{y\}$.  

\smallskip
\textit{(i)} If $a=0$, then $\rho(\mathcal{Z})=\langle\mathbf{F}_{|\ell}\rangle$ by \eqref{eq:im Z}. 
Moreover, the transitivity of the action of $H$ implies that none of the polynomials $F_{i|\ell}$ vanishes outside $y$, i.e. $\mathbf{F}_{|\ell}=\left(\alpha_1X_1^{d_1},\dots,\alpha_cX_1^{d_c}\right)$ for some $\alpha_1,\dots,\alpha_c\in \mathbb{C}$.
Finally, since $\left.X_1\frac{\partial \mathbf{F}}{\partial X_1}\right|_{\ell}=\left(d_1\alpha_1X_1^{d_1},\dots,d_c\alpha_cX_1^{d_c}\right)\in \langle\mathbf{F}_{|\ell}\rangle$, we conclude that there exists an integer $r\geq 1$ such that $\left.X_1\frac{\partial \mathbf{F}}{\partial X_1}\right|_{\ell}= r\mathbf{F}_{|\ell}$ and $\alpha_i=0$ for any $1\leq i\leq c$ with $d_i\neq r$.
Thus $\psi\left(y,\mathbf{F}\right)=\left(y,[\ell],\mathbf{F}\right)\in \widetilde{\Delta}_{r}$, as claimed.

\smallskip
\textit{(ii)} If $a=1$, then $\dim \big\langle \,\mathbf{F}_{|\ell}, \left.X_1\frac{\partial \mathbf{F}}{\partial X_0}\right|_{\ell}, \left.X_1\frac{\partial \mathbf{F}}{\partial X_1}\right|_{\ell}\big\rangle\leq 2$ by \eqref{eq:im Z}.
We point out that, if $F_{i|\ell}$ were a non-null polynomial vanishing at two distinct points of $\ell\smallsetminus\{y\}$ for some $1\leq i\leq c$, then the orbit of $\mathbf{F}_{|\ell}$ under the doubly-transitive action of $H_y$ would describe a subspace of dimension at least 4.
Thus, if $F_{i|\ell}\neq 0$, there exists at most one point $y'\in \ell\smallsetminus\{y\}$ such that $\mathbf{F}_{i|\ell}(y')=0$.
Furthermore, the very same argument shows that the point $y'$ is the same for any $i=1,\dots, c$ such that $\mathbf{F}_{i|\ell}\neq 0$.
Hence there exists a linear polynomial $L=X_0-bX_1$ vanishing at $y'=[b,1,0,\dots,0]$ such that $$\mathbf{F}_{|\ell}=\left(\alpha_1X_1^{r_1}L^{d_1-r_1},\dots,\alpha_cX_1^{r_c}L^{d_c-r_c}\right),$$
for some (possibly null) $\alpha_1,\dots,\alpha_c\in \mathbb{C}$.
To conclude, we note that $\left.X_1\frac{\partial F_i}{\partial X_1}\right|_{\ell}=r_iF_{i|\ell}-\left.bX_1\frac{\partial F_i}{\partial X_0}\right|_{\ell}$ for any $i=1,\dots,c$.
Since $\mathbf{F}_{|\ell}$, $\left.X_1\frac{\partial \mathbf{F}}{\partial X_0}\right|_{\ell}$ and $ \left.X_1\frac{\partial \mathbf{F}}{\partial X_1}\right|_{\ell}$ are linearly dependent, we conclude that $r_i$ does not depend on $i$, i.e. there exists an integer $r\geq 1$ such that $F_{i|\ell}=\alpha_iX_1^{r}L^{d_i-r}$ for any $i=1,\dots, c$ (where $\alpha_i=0$ when $r>d_i$).
Therefore, we conclude that $\psi\left(y,\mathbf{F}\right)=\left(y,[\ell],\mathbf{F}\right)\in \widetilde{\Delta}_{r,\d-r}$.
\end{proof}
\end{proposition}
\begin{proof}[Proof of Theorem \ref{thm:CLS+}]
The result follows immediately from Proposition \ref{prop:Delta_r}, item (i), and Remark \ref{rmk:mumford} by projecting $\widetilde{\Delta}_{r,\F}$ to $\mathbb P^n$.   
\end{proof}

\section{Proof of Theorem \ref{thm:main}}\label{s:proof Thm E}

In this section, we are aimed at proving Theorem \ref{thm:main}.
We will need the following folklore result to draw some conclusions on subvarieties of the contact and bi-contact loci described before. 

\begin{lemma}\label{lem:folk}
    Let $V$ be a smooth projective variety. 
\begin{enumerate}
    \item If $h^0(V,K_V)>0$, then any subvariety $W\subset V$ passing through a very general point of $V$ has $p_g(W)>0$. 
    \item If $H$ is a big and nef line bundle on $V$ such that $h^0(V,K_V-H)>0$, then 
     any subvariety $W$ passing through the very general point of $V$ verifies $h^0(\widetilde W, K_{\widetilde W} -\nu^*H)>0$, where $\nu\colon \widetilde W\longrightarrow W$ is a desingularization.
\end{enumerate}
\begin{proof}
    (1) We may assume that we have a flat family 
    $\mathcal W\longrightarrow T$, with a dominant map     
    $\mathcal W\longrightarrow T$. Notice that by adjunction the canonical divisor $K_{W_t}$ of the generic fiber $W_t$ equals $(K_{\mathcal W})_{|W_t}$. Hence the pull-back of a non-zero global section of $K_V$ yields a non-zero global section of $K_{W_t}$ and we are done.  

    (2) The proof is essentially the same. Since $h^0(V,K_V-H)>0$, by pulling-back a non zero-section in  $H^0(V,K_V-H)$ to $\widetilde W_t$, we obtain  the desired   non zero-section in  $H^0({\widetilde W_t},K_{\widetilde W_t}-\nu^*H)$. 
\end{proof}
\end{lemma}

In order to achieve Theorem \ref{thm:main}, 
we follow the argument of \cite[Theorem 2.3]{RY20} and \cite[Theorem 5.4]{CoR22}.
To this aim, we recall that given two integers $0<s\leq M$, a \emph{parameterized $s$-plane} in $\mathbb{P}^M$ is a linear map $\Lambda\in\mathrm{Hom}(\mathbb{P}^s,\mathbb{P}^M)$, which embeds $\mathbb{P}^s$ as a $s$-plane $\Lambda(\mathbb{P}^s)$ of $\mathbb{P}^M$. Moreover, given a line $\ell\subset \mathbb{P}^M$, we denote by $G_{\ell}(s,M)$ the space of parameterized $s$-planes $\Lambda$ in $\mathbb{P}^M$ such that $\ell\subset \Lambda(\mathbb{P}^s)$.
Then the following holds (cf. \cite[Corollary 2.2]{RY22} and \cite[proof of Theorem 5.4]{CoR22}).

\begin{proposition}\label{prop:RY}
    Consider a line $\ell\subset \mathbb{P}^M$.
    Let $C\subset G_{\ell}(s-1,M)$ be a non-empty subvariety of codimension $\epsilon\geq 1$, and let $B\subset G_{\ell}(s,M)$ be the subvariety of parameterized $s$-planes containing some parameterized $(s-1)$-plane of $C$. 
    Then 
    $$\codim_{G_{\ell}(s,M)}B\leq \epsilon -1.$$
\end{proposition}

We can finally prove Theorem \ref{thm:main}.

\begin{proof}[Proof of Theorem \ref{thm:main}]
Let $X\subset \mathbb{P}^n$ be a very general complete intersection of type $\mathbf{d}=(d_1,\dots,d_c)$ satisfying \eqref{eq:bound lineare} and \eqref{eq:bound quadratico}, and let $\mathbf{F}=(F_1,\dots,F_c)\in S^{\mathbf{d}}$ such that $X=X_{\mathbf{F}}=V(F_1,\dots,F_c)$.

\smallskip
\textbf{Case $a=0$.} To start, we discuss case $a=0$, and we consider a $k$-dimensional subvariety $Y=Y_{\mathbf{F}}\subset X$ such that $h^0(\widetilde Y,K_{\widetilde Y})=0$, where $\widetilde{Y}\longrightarrow Y$ is any desingularization of $Y$.
Let $\psi\colon \mathcal{Y}  \dashrightarrow \Delta$ be the map defined in \eqref{eq:psi}. 
Thanks to Proposition \ref{prop:Delta_r}, there exists a positive integer $r$ such that
$$
\psi(Y\times \{\mathbf{F}\})\subset \widetilde{\Delta}_{r, \mathbf{F}}\coloneqq \left\{(x,[\ell],\mathbf{F})\in \Delta\left|\,\ell\cdot V(F_i)\geq rx \text{ if }d_i=r \text{ and } \ell\subset V(F_i)\text{ otherwise}\right.\right\}.
$$
We note that if $r\not\in \{d_1,\dots ,d_c\}$, then any triple $(x,[\ell],\mathbf{F})\in \widetilde{\Delta}_{r, \mathbf{F}}$ is such that $\ell\subset X$, and the assertion follows because $Y\subset X$ is contained in the locus covered by the lines of $X$.
So we assume hereafter that $r\in \{d_1,\dots ,d_c\}$, that is 
$$\alpha\coloneqq \#\{i\,|\,d_i=r\}\geq 1.$$
We recall from \eqref{eq:dim Delta_r} that
\begin{equation}\label{eq:t}
t\coloneqq \dim \widetilde{\Delta}_{r, \mathbf{F}}= 
2n-1-d-c+\alpha,
\end{equation}
and from Lemma \ref{lem:cancontact} that
\begin{equation}\label{eq:omega Sigma_r}
K_{\widetilde{\Delta}_{r,\mathbf{F}}}= (\alpha r -2)H+\left( \sum_{i=1}^c \frac{d_i(d_i+1)}{2} - \alpha r -n\right) L,
\end{equation}
where $\alpha r-2\geq 0$, as $\alpha\geq 1$ and $r=d_i\geq 2$ for some $i=1,\dots,c$.
Setting 
\begin{equation}\label{eq:m}
m\coloneqq \sum_{i=1}^c \frac{d_i(d_i+1)}{2} - \alpha r,
\end{equation}
we note that 
\begin{equation}\label{eq:m-n}
    m-n-t+k= 
\sum_{i=1}^c \frac{d_i(d_i+1)}{2} -\alpha r -3n  +d +c -\alpha +k +1 \geq 0
\end{equation}
as $d-\alpha r \geq 0$, $c-\alpha\geq 0$ and \eqref{eq:bound quadratico} holds with $a=0$. 

\smallskip
For any $i=1,\dots,c$, let $\mathbb{H}^{d_i}$ be the Hilbert scheme parameterizing hypersurfaces of degree $d_i$ in $\mathbb{P}^{n}$.
Given $\mathbf{G}=(G_1,\dots,G_c)\in S^{\mathbf{d}}$, we set $[\mathbf{G}]\coloneqq\left([V(G_1)],\dots,[V(G_c)]\right)\in \mathbb{H}^{d_1}\times \dots \times \mathbb{H}^{d_c}$, and we consider the open subset $\mathbb{H}^{\mathbf{d}}_n\subset \mathbb{H}^{d_1}\times \dots \times \mathbb{H}^{d_c}$ parameterizing $c$-tuples $[\mathbf{G}]$ such that $V(G_1,\dots,G_c)\in \mathbb{P}^n$ is a smooth complete intersection.
Then we define 
$$\mathcal{U}_n^{\mathbf{d}}\coloneqq\left\{\left.\left([\ell],[\mathbf{G}]\right)\in \mathbb{G}(1,n)\times \mathbb{H}^{\mathbf{d}}_n\right|\exists x\in \ell \text{ such that }(x,[\ell],\mathbf{G})\in \widetilde{\Delta}_{r,\mathbf{G}}\right\},$$
and we consider the sublocus 
$$
\mathcal{R}_n^{\mathbf{d}}\coloneqq\left\{([\ell],[\mathbf{G}])\in \mathcal{U}_n^{\mathbf{d}}\left|\begin{array}{l}\exists x\in \ell \text{ and } \exists W\subset \widetilde{\Delta}_{r,\mathbf{G}} \text{ with desingularization } \widetilde{W}\longrightarrow W \\ 
 \text{such that } (x,[\ell],\mathbf{G}) \in W,\ \dim W=k,  \text{ and } p_g({\widetilde W})=0    \end{array} \right.\right\}.
$$
Thanks to \eqref{eq:m-n}, we have that $m-n\geq 0$, and hence $K_{\widetilde{\Delta}_{r, \mathbf{F}}}$ is effective. 
Therefore, Lemma \ref{lem:folk} ensures that no $k$-dimensional subvariety $W\subset \widetilde{\Delta}_{r, \mathbf{F}}$ passing through a very general point admits a desingularization $\widetilde{W}\longrightarrow W$ such that $p_g({\widetilde W})=0$.
Using this fact, we deduce by standard arguments that $\mathcal{R}_n^{\mathbf{d}}$ is the union of (at most) countably many locally closed irreducible subvarieties of $\mathcal{U}_n^{\mathbf{d}}$ (see e.g. \cite[p.\, 642]{BCFS23}). 
Actually, using \eqref{eq:omega Sigma_r}, we see that the same is true for any $\mathcal{R}_s^{\mathbf{d}}$ with $s\leq m$.

\smallskip

Let $y\in Y$ be a general point and let $\psi(y,\mathbf{F})\coloneqq (y, [\ell_0], \mathbf{F})$.
We assume by contradiction that $y$ lies outside the locus of $X$ covered by lines.
Since $\psi(Y\times \{\mathbf{F}\})\subset \widetilde{\Delta}_{r, \mathbf{F}}$ is birational to $Y$, then  $\psi(Y\times \{\mathbf{F}\})$ is a $k$-dimensional subvariety of $\widetilde{\Delta}_{r, \mathbf{F}}$ which passes through $(y, [\ell_0], \mathbf{F})$ and admits a desingularization with vanishing geometric genus.
Hence $([\ell_0],[\mathbf{F}])\in \mathcal{R}_n^{\mathbf{d}}$, and the fibre over $[\mathbf{F}]$ of the projection $\pi_2\colon \mathcal{R}_n^{\mathbf{d}}\longrightarrow \mathbb{H}_n^{\mathbf{d}}$ has dimension $\dim \pi_2^{-1}(\mathbf{F})\geq k$.
Therefore, 
\begin{equation}\label{eq:RY1}
    \codim_{\mathcal{U}_n^{\mathbf{d}}} \mathcal{R}_n^{\mathbf{d}}=\left(\dim \mathbb{H}_n^{\mathbf{d}}+ t\right)-\left(\dim \mathbb{H}_n^{\mathbf{d}}+ \dim \varphi_2^{-1}([\mathbf{F}])\right)\leq t-k.
\end{equation}

Let us now consider a very general point $([\ell'],[\mathbf{F}'])\in \mathcal{U}_m^{\mathbf{d}}$ and let $(x',[\ell'],[\mathbf{F}'])\in \widetilde{\Delta}_{r, \mathbf{F}'}$.
Since $K_{\widetilde{\Delta}_{r, \mathbf{F}'}}$ is effective, there are not $k$-dimensional subvarieties of $\widetilde{\Delta}_{r,\mathbf{F}'}$ which pass through $(x',[\ell'],[\mathbf{F}'])$ and admit a desingularization with vanishing geometric genus.
Thus $([\ell'],\mathbf{F}')\in \mathcal{U}_m^{\mathbf{d}}\smallsetminus \mathcal{R}_m^{\mathbf{d}}$.

Arguing as in \cite[Theorem 2.3]{RY22}, we may fix an integer $M\gg m>n$ and a pair $([\ell''],[\mathbf{F}''])\in \mathcal{U}_M^{\mathbf{d}}$ such that $X=X_{\mathbf{F}}\subset \mathbb{P}^n$ and $X'=X_{\mathbf{F}'}\subset \mathbb{P}^m$ are linear sections of $X'=X_{\mathbf{F}'}\subset \mathbb{P}^M$ by a $n$-plane $H$ and a $m$-plane $H'$ respectively, with $\ell_0=\ell'=\ell''$. 

Now, for any $s\geq n$, let $\mathcal{Z}_s \subset G_{\ell_0}(s,M)$ be the set of parameterized $s$-planes containing the line $\ell_0$, so that an element $\Lambda\colon \mathbb{P}^s \longrightarrow \mathbb{P}^M$ of $\mathcal{Z}_s$ embeds $\mathbb{P}^{s}$ as a $s$-plane $\Lambda(\mathbb{P}^{s})$ containing $\ell_0$.
Moreover, let us fix homogeneous coordinates $[y_0,\dots,y_n]$ in $\mathbb{P}^{s}$ and, given a parameterized plane $\Lambda\colon \mathbb{P}^s \longrightarrow \mathbb{P}^M$, let us consider the pair $([\ell_{\Lambda}],[\mathbf{F}''_{\Lambda}])\in \mathcal{U}_s^{\mathbf{d}}$, where 
$$\ell_\Lambda\coloneqq\Lambda^{-1}(\ell_0) \quad \text{and} \quad [\mathbf{F}''_{\Lambda}]\coloneqq\left(\left[\Lambda^{-1}\left(V(F''_1)\cap \Lambda(\mathbb{P}^{s})\right)\right],\dots,\left[\Lambda^{-1}\left(V(F''_c)\cap \Lambda(\mathbb{P}^{s})\right)\right]\right),$$
i.e. $[\mathbf{F}''_{\Lambda}]$ describes the section $X''\cap \Lambda(\mathbb{P}^{s})$ as a complete intersection of type $\mathbf{d}$ in $\mathbb{P}^{s}$.  
Then we may define a morphism $\phi_s\colon \mathcal{Z}_s\longrightarrow \mathcal{U}_s^{\mathbf{d}}$ sending $\Lambda$ to $([\ell_{\Lambda}],[\mathbf{F}''_{\Lambda}])$. 
Let $\mathcal{F}_s$ be its image, and let $\mathcal{Z}'_s\coloneqq \phi_s^{-1}(\mathcal{F}_s\cap \mathcal{R}_s^{\mathbf{d}})\subset \mathcal{Z}_s$ be the subset of parameterized $s$-planes $\Lambda\subset \mathbb{P}^M$ such that $([\ell_{\Lambda}],[\mathbf{F}''_{\Lambda}])\in \mathcal{R}_s^{\mathbf{d}}$.

We point out that 
\begin{equation}\label{eq:Z'_s}
\codim_{\mathcal{U}_s^{\mathbf{d}}} \mathcal{R}_s^{\mathbf{d}}\geq \codim_{\mathcal{F}_s} \left(\mathcal{F}_s\cap \mathcal{R}_s^{\mathbf{d}}\right)\geq \codim_{\mathcal{Z}_s} \mathcal{Z}'_s.
\end{equation}
We claim that for any $0\leq q\leq m$, we have that $\codim_{\mathcal{Z}_{m-q}} \mathcal{Z}'_{m-q}\geq q+1$.
When $q=0$, we recall that $([\ell'],\mathbf{F}')\in \mathcal{U}_m^{\mathbf{d}}\smallsetminus \mathcal{R}_m^{\mathbf{d}}$, so that $\phi_m^{-1}([\ell'],\mathbf{F}')\in \mathcal{Z}_{m}\smallsetminus \mathcal{Z}'_{m}$ and $\codim_{\mathcal{Z}_{m}} \mathcal{Z}'_{m}\geq 1$.
By induction on $q\geq 0$, suppose that $\codim_{\mathcal{Z}_{m-q+1}} \mathcal{Z}'_{m-q+1}\geq q$.
We use Proposition \ref{prop:RY} with $C\coloneqq \mathcal{Z}'_{m-q}$. 
Accordingly, let $B\subset G_{\ell_0}(m-q+1,M)$ be the subvariety of parameterized $(m-q+1)$-planes containing some $[\Lambda]\in C$, and notice that $B\subset\mathcal{Z}'_{m-q+1}$.
Thus
$$
    q\leq \codim_{\mathcal{Z}_{m-q+1}} \mathcal{Z}'_{m-q+1} \leq \codim_{\mathcal{Z}_{m-q+1}} B \leq \codim_{\mathcal{Z}_{m-q}} \mathcal{Z}'_{m-q} -1,
$$
which proves the claim.

Therefore, setting $q=m-n$, and using inequalities \eqref{eq:m-n} and \eqref{eq:Z'_s}, we obtain
$$
\codim_{\mathcal{U}_n^{\mathbf{d}}} \mathcal{R}_n^{\mathbf{d}}\geq \codim_{\mathcal{Z}_n} \mathcal{Z}'_n\geq m-n+1\geq t-k+1,
$$
which contradicts \eqref{eq:RY1}.

Thus, the general point $y\in Y$ must be contained in the locus of $X$ covered by the lines contained in $X$, which proves assertion for $a=0$.

\smallskip
\textbf{Case $a=1$.} The proof in this  case is analogous.
We summarize the main steps of the argument.
Let  $Y=Y_{\mathbf{F}}\subset X\subset \mathbb{P}^n$ be a $k$-dimensional subvariety such that $h^0(\widetilde Y, K_{\widetilde Y}-\nu^*H)=0$, where $\nu\colon\widetilde{Y}\longrightarrow Y$ is any desingularization, and $H\coloneqq \mathcal O_{X}(1)$ is the hyperplane bundle.
Thanks to Proposition \ref{prop:Delta_r}, there exists a positive integer $r$ such that
$$
\psi(Y\times \{\mathbf{F}\})\subset {\widetilde \Delta}_{r,\mathbf{d}-r, \mathbf{F}}\coloneqq \left\{(x,[\ell],\mathbf{F})\left|\begin{array}{c} \exists x'\in \ell \text{ such that } \forall i=1,\dots,c \\ \ell\cdot V(F_i)\geq rx+(d_i-r)x' \end{array}\right.\right\}.
$$
If $r>d_c$, then any triple $(x,[\ell],\mathbf{F})\in {\widetilde\Delta}_{r,\mathbf{d}-r, \mathbf{F}}$ is such that $\ell\subset X$, and the assertion follows because $Y\subset X$ is contained in the locus covered by the lines of $X$.
So we assume hereafter that $1\leq r\leq d_c$, and we set 
$$\gamma\coloneqq \#\{i\,|\,d_i=r\} \quad\text{and}\quad \delta\coloneqq \#\{i\,|\,d_i> r\}, \quad\text{with}\quad \beta\coloneqq\gamma+\delta\geq 1.$$

Let $y\in Y$ be a general point and let $\psi(y,\mathbf{F})\coloneqq (y, [\ell_0], \mathbf{F})$.
We assume by contradiction that $y$ lies outside the locus of $X$ covered by lines.
Moreover, 
we distinguish three cases:
$$
\text{(a) }\beta r \geq 3 \ \text{ and }\ \sum_{d_i\geq r}(d_i-r)\geq 2, \quad \text{(b) }\sum_{d_i\geq r}(d_i-r)\leq 1, \quad \text{(c) }\beta r \leq 2.
$$

\smallskip
(a) Consider the locus ${\widetilde\Lambda}_{r, \mathbf{G}}$ defined in \eqref{eq:Lambda1}, i.e.
$$
{\widetilde\Lambda}_{r, \mathbf{G}}=\overline{\left\{(x,x')\in \mathrm{Bl}_{\textrm{Diag}}(\mathbb{P}^n\times \mathbb{P}^n)\left|\begin{array}{c}\forall i=1,\dots,c \text{ the line }\ell\coloneqq \langle x,x'\rangle \text{ satisfies }\\ \ell\cdot V(G_i)\geq rx+(d_i-r)x' \end{array} \right.\right\}}.
$$
Now, we define 
$$
\mathcal{U}_n^{\mathbf{d}}\coloneqq\left\{\left.\left([\ell],[\mathbf{G}]\right)\in \mathbb{G}(1,n)\times \mathbb{H}^{\mathbf{d}}_n\right|\exists x,x'\in \ell \text{ such that }(x,x')\in {\Lambda}_{r,\mathbf{G}}\right\},
$$
and we consider the sublocus 
$$
\mathcal{R}_n^{\mathbf{d}}\coloneqq\left\{([\ell],[\mathbf{G}])\in \mathcal{U}_n^{\mathbf{d}}\left|\begin{array}{l}\exists W\subset {\widetilde\Lambda}_{r,\mathbf{G}} \text{ such that } (x,x') \in W \text{ and } h^0(\widetilde W, K_{\widetilde W}-\mu^*H_1)=0,\\
\text{where }\mu\colon\widetilde{W}\longrightarrow W \text{ is a desingularization}\end{array}\right.\right\},
$$
where $H_1\coloneqq \pi_1^* O_{X}(1)$ is the pullback of the hyperplane bundle under the natural map $\pi_1\colon \mathrm{Bl}_{\textrm{Diag}}(\mathbb{P}^n\times \mathbb{P}^n)\longrightarrow \mathbb{P}^n$ induced by the first projection.

Since $y$ lies outside the locus of $X$ covered by lines, we have a rational map $p\colon Y\dashrightarrow {\Lambda}_{r, \mathbf{F}}$ which sends a general point $x\in Y\subset X_{\mathbf{F}}$ to the unique point $(x,x')$ such that $\psi(x, \mathbf{F})=(x,[\ell],\mathbf{F})$ and $\ell\cdot V(F_i)\geq rx+(d_i-r)x'$ for any $i=1,\dots,c$.
In particular, considering the Zariski closure $Z\coloneqq \overline{p(Y)}$, we have that $p\colon Y\dashrightarrow Z$ is a birational map, whose inverse is the restriction to $Z$ of the morphism $\pi_1\colon \mathrm{Bl}_{\textrm{Diag}}(\mathbb{P}^n\times \mathbb{P}^n)\longrightarrow \mathbb{P}^n$.

We point out that by construction $([\ell_0],[\mathbf{F}])\in \mathcal{R}_n^{\mathbf{d}}$.
Indeed, denoting by $y'\in \ell_0$ the point such that $\ell_0\cdot V(F_i)\geq ry+(d_i-r)y'$ for any $i=1,\dots,c$, we have that $(y,y')\in Z\subset \widetilde\Lambda_{r,\mathbf{F}}$,
and given a desingularization $\mu\colon \widetilde{Z}\longrightarrow Z$, the composition $(\mu\circ \pi_{1|Z})\colon \widetilde{Z} \longrightarrow Y$ is a desingularization of $Y$, so that $h^0(\widetilde Z, K_{\widetilde Z}-\mu^*H_1)=h^0(\widetilde Z, K_{\widetilde Z}-(\mu\circ \pi_{1|Z})^*H)=0$.

We set
$$
t\coloneqq \dim {\widetilde\Lambda}_{r, \mathbf{F}}= 
2n-d-(c-\beta) \quad\text{and}\quad m\coloneqq 1+ \sum_{i=1}^c\binom{d_i+1}{2} - \sum_{d_i\geq r}d_i.
$$
Therefore, we deduce from \eqref{eq:bound quadratico} that
\begin{equation}\label{eq:m-n 2}
    m-n-t+k= 
1+\sum_{i=1}^c \frac{d_i(d_i+1)}{2} - \sum_{d_i\geq r}d_i -3n  +d +c -\beta +k \geq 0
\end{equation}
and, in particular, $m-n\geq 0$.
Thus it follows from assumption (a) and Lemma \ref{lem:can-bicontact} that $K_{{\widetilde\Lambda}_{r, \mathbf{F}}}-H_1$ is effective. 
Then Lemma \ref{lem:folk} ensures that there are not $k$-dimensional subvarieties $W\subset {\widetilde\Lambda}_{r, \mathbf{F}}$ passing through a very general point of ${\widetilde\Lambda}_{r, \mathbf{F}}$ and satisfying $h^0(\widetilde W, K_{\widetilde W}-\nu^*H_1)=0$.

By arguing as above, we note that $\mathcal{R}_s^{\mathbf{d}}$ is the union of (at most) countably many locally closed irreducible subvarieties of $\mathcal{U}_s^{\mathbf{d}}$ for any $s\leq m$. Therefore, by following the very same argument as case $a=0$, we obtain a contradiction.

\smallskip
(b) We point out that $\sum_{d_i\geq r}(d_i-r)\leq 1$ if and only if $d_1,\dots,d_{c-1}\leq r$ and $d_c\in\left\{r,r+1\right\}$.

\smallskip
Suppose that $d_c=r+1$. 
In the universal line $\Gamma\subset \mathbb{P}^n\times\mathbb{G}(1,n)$, we consider the locus ${\widetilde\Lambda}_{r, \mathbf{G}}$ defined in \eqref{eq:Lambda2}, i.e.
$$
{\widetilde\Lambda}_{r, \mathbf{G}}=\left\{\left.(x,[\ell])\in \Gamma\right|\ell\cdot V(G_i)\geq rx\ \forall i=1,\dots,c\right\}.
$$
Moreover, we define
$$
\mathcal{U}_n^{\mathbf{d}}\coloneqq\left\{\left.\left([\ell],[\mathbf{G}]\right)\in \mathbb{G}(1,n)\times \mathbb{H}^{\mathbf{d}}_n\right|\exists x\in \ell \text{ such that }(x,[\ell])\in {\Lambda}_{r,\mathbf{G}}\right\},
$$
and we consider the sublocus 
$$
\mathcal{R}_n^{\mathbf{d}}\coloneqq\left\{([\ell],[\mathbf{G}])\in \mathcal{U}_n^{\mathbf{d}}\left|\begin{array}{l}\exists W\subset {\widetilde\Lambda}_{r,\mathbf{G}} \text{ such that } (x,[\ell]) \in W \text{ and } h^0(\widetilde W, K_{\widetilde W}-\nu^*H)=0,\\
\text{where }\nu\colon\widetilde{W}\longrightarrow W \text{ is a desingularization}\end{array}\right.\right\},
$$
where $H\coloneqq \pi_1^* O_{X}(1)$ is the pullback of the hyperplane bundle under the first projection $\pi_1\colon {\widetilde\Lambda}_{r,\mathbf{G}}\longrightarrow \mathbb{P}^n$.

Since $y$ lies outside the locus of $X$ covered by lines, we have a rational map $p\colon Y \dashrightarrow {\Lambda}_{r, \mathbf{F}}$ sending a general point $(x, \mathbf{F})$ to $(x,[\ell_{(x,\mathbf{F})}])$.
As in case (a), this map is birational to the closure of its image $Z\coloneqq \overline{p(Y)}$, and its inverse is given by $\pi_{1|Z}\colon Z\longrightarrow \mathbb{P}^n$.
Therefore, the same argument shows that $([\ell_0],[\mathbf{F}])\in \mathcal{R}_n^{\mathbf{d}}$.

Then, we set
$$
t\coloneqq \dim {\widetilde\Lambda}_{r, \mathbf{F}}= 
2n-d-(c-\beta) \quad\text{and}\quad m\coloneqq 1+\sum_{i=1}^c\binom{d_i+1}{2} - \sum_{d_i\geq r}d_i - d_c.
$$
and we consider the canonical bundle of ${\widetilde\Lambda}_{r,\mathbf{G}}$, which is governed by Lemma \ref{lem:can-bicontact-bis}.
We point out that $2\delta r=2r\geq 4$; indeed, if $r=1$, then $c=1$ and $d_c=2$ because we are assuming $2\leq d_1,\dots, d_{c-1}=r$ and $d_c=r+1$, but by assumption $d>2$.
We deduce from this fact and \eqref{eq:bound quadratico} that $K_{{\widetilde\Lambda}_{r,\mathbf{G}}}-H$ is effective. 
Therefore, we conclude by using the very same argument as above.

\smallskip
If instead $d_c=r$, then $\psi(Y\times \{\mathbf{F}\})\subset \widetilde{\Delta}_{r, \mathbf{F}}$ as in case $a=0$.
So, we set
$$
\mathcal{U}_n^{\mathbf{d}}\coloneqq\left\{\left.\left([\ell],[\mathbf{G}]\right)\in \mathbb{G}(1,n)\times \mathbb{H}^{\mathbf{d}}_n\right|\exists x\in \ell \text{ such that }(x,[\ell])\in \widetilde{\Delta}_{r,\mathbf{G}}\right\},
$$
and we introduce the sublocus 
$$
\mathcal{R}_n^{\mathbf{d}}\coloneqq\left\{([\ell],[\mathbf{G}])\in \mathcal{U}_n^{\mathbf{d}}\left|\begin{array}{l}\exists W\subset \widetilde{\Delta}_{r,\mathbf{G}} \text{ such that } (x,[\ell]) \in W \text{ and } h^0(\widetilde W, K_{\widetilde W}-\nu^*H)=0,\\
\text{where }\nu\colon\widetilde{W}\longrightarrow W \text{ is a desingularization}\end{array}\right.\right\},
$$
where $H\coloneqq \pi_1^* O_{X}(1)$ is the pullback of the hyperplane bundle under the first projection $\pi_1\colon \widetilde{\Delta}_{r,\mathbf{G}}\longrightarrow \mathbb{P}^n$.
Then we define $t$ and $m$ as in \eqref{eq:t} and \eqref{eq:m}, and we note that $\alpha r \geq 3$.
Indeed $r=d_c\geq 2$, and $\alpha$ must be greater than 1, otherwise $X_{\mathbf{F}}$ would be a quadric hypersurface.
Combining this fact and \eqref{eq:bound quadratico}, we obtain that $K_{\widetilde{\Delta}_{r,\mathbf{G}}}-H$ is effective.
Again, we conclude using the usual argument.


\smallskip
(c) Finally, let us assume that $\beta r \leq 2$.
Since $d_c\geq \dots\geq d_1\geq 2$, this case occurs if and only if $\beta=r=c=1$, or $\beta=c=1$ and $r=2$, or $\beta=c=2$ and $r=1$.  
In this case, we consider the loci
$$
\widetilde\Lambda_{r,\mathbf{d}-r, \mathbf{G}} \coloneqq\{(x,x',[\ell])| \ell\cdot V(G_i)\geq rx+(d_i-r) x' \ \forall i=1,\dots,c\}
\subset \mathbb P^n\times \mathbb P^n\times \mathbb G(1,n),
$$
$$
\mathcal{U}_n^{\mathbf{d}}\coloneqq\left\{\left.\left([\ell],[\mathbf{G}]\right)\in \mathbb{G}(1,n)\times \mathbb{H}^{\mathbf{d}}_n\right|\exists x,x'\in \ell \text{ such that }(x,x',[\ell])\in {\widetilde\Lambda}_{r,\mathbf{d}-r,\mathbf{G}}\right\},
$$
and we introduce the sublocus $\mathcal{R}_n^{\mathbf{d}}\subset \mathcal{U}_n^{\mathbf{d}}$ analogously.
Moreover, we set 
$$
t\coloneqq \dim {\widetilde\Lambda}_{r,\mathbf{d}-r, \mathbf{F}}= 
2n-d \quad\text{and}\quad m\coloneqq 1+\sum_{d_i\geq r}\frac{r(r-1)}{2} + \sum_{d_i> r}\frac{(d_i-r)(d_i-r-1)}{2}.
$$
It follows from Lemma \ref{lem:can-bicontact-ter} that 
$$
K_{\widetilde\Lambda_{r,\mathbf{d}-r,\F}} = \left(-2 + \sum_{d_i\geq r} r(d_i-r+1)\right) H_1 + \left(-2 + \sum_{d_i> r} (d_i-r)(r+1)\right) H_2 + \left(m-n \right) L .    
$$
Since we are in case (c), using \eqref{eq:bound quadratico} and the the fact that $X_{\F}$ is not a quadric hypersurface, it is easy to check that $m-n-t+k\geq 0$ and $K_{\widetilde\Lambda_{r,\mathbf{d}-r,\F}}-H_1$ is effective.
Thus we may obtain a contradiction by arguing as above.

\smallskip
In conclusion, if $a=1$, any $k$-dimensional variety $Y\subset X$ such that $h^0(\widetilde Y, K_{\widetilde Y}-\nu^*H)=0$ is contained in the locus covered by the lines in $X$.
\end{proof}

\begin{remark}\label{rem:(1.2-0)}
    We point out the when $a=0$, the bound \eqref{eq:bound quadratico} can be slightly weakened.
    Indeed, its role in the proof of Theorem \ref{thm:main} is to guarantee that $m-n-t+k\geq 0$ as in \eqref{eq:m-n}.
    Therefore, setting 
    $$
    \rho\coloneqq \max_{j=1,\dots ,c} \alpha_j(d_j+1),\quad \text{where } \alpha_j\coloneqq \#\{i|d_i=d_j\} \text{ for any }j=1,\dots, c,
    $$
    we can replace \eqref{eq:bound quadratico} by
\begin{equation*}
\sum_{i=1}^c \frac{d_i(d_i+1)}{2}+ d -\rho\geq 3n-k-c-1.  
\end{equation*}
Moreover, Theorem \ref{thm:main-distinc} below shows that this inequality could be even weaker under the assumption $d_1<\dots<d_c$.      
\end{remark}

\begin{remark}\label{rem:(1.2-1)}
As far as the case $a=1$ is concerned, we note that if $c\geq 3$, we only need to discuss cases (a) and (b) above, and hence \eqref{eq:bound quadratico} can be slightly weakened as follows:
\begin{equation}\label{eq:(1.2-1)}
\sum_{i=1}^c \frac{d_i(d_i+1)}{2}-d_c\geq 3n-k-1.  
\end{equation}
\end{remark}

\smallskip
We conclude this section by discussing an improvement in the case $a=0$ when all the degrees are distinct. 
In this case
we can argue by adjunction in a way similar to \cite{CR04,P04} and obtain better bounds. Precisely we have the following. 

\begin{theorem}\label{thm:main-distinc}
Let $n,c,k$ be positive integers and let $X\subset \mathbb{P}^n$ be a very general complete intersection of multidegree $(d_1,\dots,d_c)$ with $d_i\not=d_j$ $\forall i\not=j$.
Let $Y\subset X$ be any $k$-dimensional subvariety such that $h^0(K_{\widetilde Y})=0$, where  $\widetilde{Y}\longrightarrow Y$ is any desingularization of $Y$.
If 
\begin{equation}\label{eq:bound lineare-spe}
d\geq n+1+\left\lfloor\frac{n-c-k+1}{2}\right\rfloor     
\end{equation}
 and 
\begin{equation}\label{eq:bound quadratico-bis}
\sum_{i=1}^c \frac{d_i(d_i+1)}{2}+(d-d_c)\geq 3n-c-k,  
\end{equation}
then $Y$ is contained in the union of the lines lying on $X$.
\end{theorem}

\begin{proof}
    By Proposition \ref{prop:Delta_r}, item (i), $\widetilde{Y}_\F$ injects into the contact locus
   $$\widetilde{\Delta}_{r,\F}\coloneqq \left\{(x,[\ell])\left|\,\ell\cdot V(G_i)\geq rx \text{ if }d_i=r \text{ and } \ell\subset V(G_i)\text{ otherwise}\right.\right\}.$$ 
Notice that by (\ref{eq:dim Delta_r}), under our hypothesis $d_i\not=d_j$ $\forall i\not=j$, its codimension is given by the formula
$$
\codim_{\widetilde{\Delta}_{r,\F}} (\widetilde{Y}_\F)=2n-d-c-k.
$$
Using notation as in \S \ref{sub:contact}, we have
$$
\mathcal A_{{\bf d}, r}= \Big( \bigoplus_{d_j\not=r} q^*\mathcal E_{d_j} \Big) \oplus \mathcal F_{r}.
$$
Set $\mathcal E_{\d}:=\Big( \bigoplus_{i=1}^c \mathcal E_{d_i} \Big)$.
Consider the kernel of the two natural evaluation maps
$$
0\longrightarrow \mathcal{M}_\d \longrightarrow S^\d\otimes \mathcal{O}_{\mathcal{P}}\longrightarrow q^*\mathcal E_{\d}\longrightarrow 0
$$
and 
$$
0\longrightarrow \mathcal{N}_\d \longrightarrow S^\d\otimes \mathcal{O}_{\mathcal{P}}\longrightarrow\mathcal A_{\d,r}\longrightarrow 0. 
$$
By definition,  the fiber of $\mathcal{M}_\d$ at a point $(x,[\ell],\F)$ is the graded homogeneous ideal $S^\d_\ell=\bigoplus H^0(I_\ell(d_i))$.
Then we have 
$$
0\longrightarrow \mathcal{M}_\d \longrightarrow
\mathcal{N}_\d \longrightarrow\mathcal{L}_r \longrightarrow 0. 
$$
Notice that the restriction $(\mathcal{N}_\d)_{|\widetilde \Delta_{r,\F}}$ coincides with the restriction to ${\widetilde \Delta_{r,\F}}$ of the vertical tangent space 
$T^{vert}_{\widetilde \Delta_{r}}\coloneqq\ker \Big(T_{\widetilde \Delta_{r}}\twoheadrightarrow T_{\mathcal P}\Big)$ 
with respect to the natural projection onto ${\mathcal P}$. Fix a point $(x,[\ell],\F)\in \mathcal Y$ and consider  $T:=T_{\mathcal{Y}, (x,[\ell],\F)}$. By the $\GL(n+1)$-invariance, we have 
$$
 \codim_{\widetilde{\Delta}_{r,\F}} (\widetilde{Y}_\F)= \codim_{(\mathcal{N}_\d)_{(x,[\ell],\F)}} T\cap (\mathcal{N}_\d)_{(x,[\ell],\F)}.
$$
We  now study the intersection 
$
T\cap S^\d_\ell \subset T\cap (\mathcal{N}_\d)_{(x,[\ell],\F)}. 
$
We have either
\begin{equation}\label{eq:cod}
    \codim_{S^\d_\ell} (T\cap S^\d_\ell) = \codim_{(\mathcal{N}_\d)_{(x,[\ell],\F)}} T\cap (\mathcal{N}_\d)_{(x,[\ell],\F)}=2n-c-d-k
\end{equation}
or 
\begin{equation}\label{eq:cod-1}
       \codim_{S^\d_\ell} (T\cap S^\d_\ell) \leq 2n-c-d-k -1.
\end{equation}

\smallskip
If \eqref{eq:cod} holds we claim that we reach a contradiction by adjunction under the numerical hypothesis \eqref{eq:bound quadratico-bis}. Indeed, by adjunction and restriction of K\"ahler differentials we have 
\begin{eqnarray}\label{eq:restr+adj}
 \wedge^{2n-d-c-k} {{T_{\widetilde \Delta_{r}}}}_{|{{\widetilde\Delta}_{r,\F}}}
\otimes K_{\widetilde\Delta_{r,F}}\cong
 {{\Omega}^{N+k}_{{\widetilde \Delta_{r}}}}_{|\widetilde \Delta_{r,\F}}
 \longrightarrow
 {\Omega ^{N+k}_{{\widetilde{\mathcal Y}}}}_{|{\widetilde Y_F}}
 \cong K_{\widetilde Y_F}.
 \end{eqnarray}
 By $\GL(n+1)$-invariance of $\mathcal Y$ it is sufficient to consider 
 \begin{eqnarray}\label{eq:restr+adj vert}
 \wedge^{2n-d-c-k}\ {{T^{vert}_{\widetilde \Delta_{r}}}}_{|{{\widetilde\Delta}_{r,\F}}}
\otimes K_{\widetilde\Delta_{r,F}}
\longrightarrow K_{\widetilde Y_F}.
 \end{eqnarray}
 Since in the case we are considering $\codim_{S^\d_\ell} (T\cap S^\d_\ell) =2n-c-d-k$,
we are reduced to consider 
\begin{eqnarray}\label{eq:final}
 \wedge^{2n-d-c-k}\ {(\mathcal M_\d)_{|{{\widetilde\Delta}_{r,\F}}}}
\otimes K_{\widetilde\Delta_{r,F}}
\longrightarrow K_{\widetilde Y_F}.
 \end{eqnarray}
 Now the bundle $\mathcal M_\d$ is the pull-back of the corresponding bundle on the Grassmanniann
 $$
0\longrightarrow {M}_\d \longrightarrow S^\d\otimes \mathcal{O}_{\mathbb{G}}\longrightarrow \mathcal E_{\d}\longrightarrow 0. 
 $$
By \cite[Proposition 2.2, item (ii)]{P03} the bundle $\mathcal M_\d\otimes L$ is globally generated. Therefore, using Lemma \ref{lem:cancontact} and taking global sections of the map \eqref{eq:final}, we produce a non-zero section of $K_{\widetilde Y_\F}$, which gives a contradiction as soon as 
$$
\sum_{i=1}^c \frac{d_i(d_i+1)}{2}-r-n\geq 2n-d-c-k,\  
i.e. 
\
\sum_{i=1}^c \frac{d_i(d_i+1)}{2}+(d-r)\geq 3n-c-k.  
$$
The conclusion now follows. 

\smallskip
Therefore, we assume that \eqref{eq:cod-1} holds.
In this case, we claim that $\F_{|\ell}$ vanishes and we are done. Indeed,
$$
\rk (T\cap S^\d_\ell) \geq \dim (S^\d_\ell)- (2n-c-d-k -1)= N-d-c - (2n-c-d-k -1)= N+k-2n+1.
$$
However,
$$\rk\ T\cap (\mathcal{N}_\d)_{(x,[\ell],\F)}=
\dim(\mathcal Y)-\dim(\mathcal P)=(N+k)-(2n-1).
$$
Hence $T\cap (\mathcal{N}_\d)_{(x,[\ell],\F)}= (T\cap S^\d_\ell)$ and the conclusion follows from the fact that $\F\in T$.
\end{proof}

\begin{remark}\label{rmk:nonfunziona col -1}
Unfortunately we do not see how to make a similar argument work to deal with subvarieties $Y_\F$ not of general type, at least for some degrees. Indeed in this case we should work on $\widetilde \Lambda_{r,\F}$.  Yet the dimension of $\widetilde \Lambda_{r,\F}$---and hence the $\codim_{\widetilde \Lambda_{r,\F}} (\widetilde Y_\F)$---depends on the number $\beta$ of $d_i$'s which are $\geq r$. This number varies from $1$ to $c$ and, as we cannot control for which $r$ the bi-osculation  occurs, we cannot reduce to a favorable situation, as in the previous proof, where either  \eqref{eq:cod}  or \eqref{eq:cod-1} hold.    
\end{remark}
\section{Proof of Theorem \ref{thm:Ein+}, Corollary \ref{cor:DR+} and Theorem \ref{thm:RY+}}

In this section, we use Theorem \ref{thm:main} to achieve Theorem \ref{thm:Ein+}, Corollary \ref{cor:DR+} and Theorem \ref{thm:RY+}.

\begin{proof}[Proof of Theorem \ref{thm:Ein+}]
    Let $X\subset \mathbb{P}^n$ be a very general complete intersection of type $(d_1,\dots,d_c)$, with $d\geq 2n-c-k$ and $n\geq \max\{2c+k+a, c+k+3+2a\}$.
    
    Let $a\in \{0,1\}$, and suppose that the assumptions of Theorem \ref{thm:main} are satisfied.
    So, if $Y\subset X$ were a $k$-dimensional subvariety such that $h^0(\widetilde Y, K_{\widetilde Y}\otimes\nu^*{\mathcal{O}}_{Y}(-a))=0$ for some desingularization $\nu\colon \widetilde Y\longrightarrow Y$, then $Y$ would be contained in the union of the lines lying on $X$.
    However, such a union of lines is supported on a subvariety of $X$ of dimension 
    $$
    2n-1-\sum_{i=1}^c(d_i+1)=2n-1-d+c\leq k-1<\dim Y
    $$
    (see e.g. \cite[Th\'eor\`eme 2.1]{DM}), a contradiction. 
    Thus $X$ does not contain $k$-dimensional subvarieties $Y$ satisfying $h^0(\widetilde Y, K_{\widetilde Y}\otimes\nu^*{\mathcal{O}}_{Y}(-a))=0$, and the assertion of Theorem \ref{thm:Ein+} follows.

    \smallskip
    Therefore, it suffices to show that the assumptions of Theorem \ref{thm:main} are satisfied.
    Since $d\geq 2n-c-k$ and $2a+3\leq n-c-k$, we have that
    $$
        n+1+\left\lfloor\frac{n-c-k+1}{2}\right\rfloor   + a  \leq \frac{3n+2a+3-c-k}{2}\leq \frac{4n-2c-2k}{2}\leq d,
    $$
    that is \eqref{eq:bound lineare} holds.
    
    In order to discuss \eqref{eq:bound quadratico}, we distinguish three cases:
    $$
    \text{(1) }a=0, \quad \text{(2) }a=1 \text{ and } c\leq 2, \quad \text{(3) }a=1 \text{ and } c\geq 3.
    $$

    \smallskip
    (1) We point out that, if $x_1+\dots+x_c=d$ is fixed, the function $f(x_1,\dots,x_c)=\sum_{i=1}^cx_i^2$ is minimal when $x_1=\dots=x_c=\frac{d}{c}$.
    Hence
   \begin{equation}\label{eq:minorazione}\sum_{i=1}^c \frac{d_i(d_i+1)}{2}=\frac{1}{2}\sum_{i=1}^cd_i^2+ \frac{1}{2}d\geq 
    \frac{1}{2c}\left(d^2+cd\right).
    \end{equation}
    Therefore, in order to check \eqref{eq:bound quadratico}, it is enough to prove that
    \begin{equation}\label{eq:caso (1)}
        \frac{1}{2c}\left(d^2+cd\right) \geq 3n-k-1.
    \end{equation}
    Using $d\geq 2n-c-k$ and $n-2c-k-a\geq 0$ with $a=0$, we obtain 
    \begin{align*}
        d^2+cd - 2c(3n-k-1) &\geq (2n-c-k)^2 + c(2n-c-k)- 2c(3n-k-1)\\ 
        &= 4n(n -2c -k) +k^2 +3ck+2c \geq 0,
    \end{align*}
    so that \eqref{eq:caso (1)} holds.

    \smallskip
    (2) Analogously, if $c=2$, then $d_1^2+d_2^2\geq 2(d/2)^2$, $d\geq 2n-2-k$ and $n\geq c+k+3+2a=k+7$.
    Hence \eqref{eq:bound quadratico} holds, as
    \begin{align*}
        \sum_{i=1}^2 \frac{d_i(d_i+1)}{2}-(d_1+d_2)-(3n-k-1) 
        & \geq \frac{1}{4}d^2- \frac{1}{2}d -3n+k+1\\
        & \geq \frac{1}{4}(2n-2-k)^2 - \frac{1}{2}(2n-2-k)- 3n+k+1\\ 
        &= n(n -6 -k) +\frac{1}{4}(k^2 +10k+12) \geq 0.
    \end{align*}
    If instead $c=a=1$, then \eqref{eq:bound quadratico} is $\frac{d(d-3)}{2}\geq 3n-k-1$, with $d\geq 2n-1-k$ and $n\geq k+6$.
    Thus
    \begin{align*}
        \frac{d(d-3)}{2} -(3n-k-1) & \geq \frac{1}{2}(2n-1-k)^2 -\frac{3}{2} (2n-1-k)- 3n+k+1 \\
        & = 2n(n -k -4) +\frac{1}{2}(k^2 +7k+6) \geq 0.
    \end{align*}
    
    \smallskip
    (3) Finally, we assume that $a=1$ and $c\geq 3$.
    In the light of Remark \ref{rem:(1.2-1)}, it is enough to check \eqref{eq:(1.2-1)}. 
    We note that, when $x_1+\dots+x_c=d$ is constant, the function $f(x_1,\dots,x_c)=\sum_{i=1}^c \frac{x_i(x_i+1)}{2}-x_c$ is minimal at $x_1=\dots=x_{c-1}=\frac{d-1}{c}$ and $x_c=x_1+1=\frac{d-1+c}{c}$.
    Therefore, 
    \begin{align*}
        \sum_{i=1}^c \frac{d_i(d_i+1)}{2}-d_c 
        &= \frac{1}{2}\left(\sum_{i=1}^{c-1} d_i^2+ d_c^2+d-2d_c\right)\\
        & \geq \frac{1}{2}\left((c-1)\frac{(d-1)^2}{c^2}+ \frac{(d-1+c)^2}{c^2} + d -2\frac{(d-1+c)}{c}\right)\\
        & = \frac{1}{2c}\left(d^2+ (c-2)d-c+1\right).
    \end{align*}
    Then, in order to achieve \eqref{eq:(1.2-1)}, it is enough to prove that $\frac{1}{2c}\left(d^2+ (c-2)d-c+1\right)\geq 3n-k-1$.
    To this aim, we recall that $d\geq 2n-c-k$ and $n-2c-k-a\geq 0$ with $a=1$, so that
    \begin{align*}
        (d^2+ (c-2)d-c+1) - 2c(3n-k-1) &\geq (2n-c-k)^2 + (c-2)(2n-c-k)- 2c(3n-k-1)\\ 
        &= 4n(n -2c -k-1) +k^2 +3ck+4c+2k \geq 0,
    \end{align*}
    and hence \eqref{eq:(1.2-1)} holds. 
\end{proof}

\begin{proof}[Proof of Corollary \ref{cor:DR+}]
 Let $X\subset \mathbb{P}^n$ be a very general complete intersection of type $(d_1,\dots,d_c)$, with $d\geq 2n-c-k$.
  Given a subvariety $Y\subset X$ and a desingularization $\nu\colon\widetilde{Y}\longrightarrow Y$, we point out that if $h^0 \left(\widetilde{Y}, K_{\widetilde{Y}}\otimes \nu^*\mathcal O_{{Y}}(-1)\right)>0$, then $Y$ is of general type.
 Indeed $E=K_{\widetilde{Y}}\otimes \nu^*\mathcal O_{{Y}}(-1)$ is effective and $\mathcal O_{{X}}(1)$ is ample, so that $K_{\widetilde{Y}}=E\otimes \nu^*\mathcal O_{{Y}}(1)$ is big, i.e. $Y$ is of general type.
 
 It follows from $d\geq 2n-c-k$ and \cite[Theorem 2.1]{E88} that for any $3\leq k\leq n-c$, all the $k$-dimensional subvarieties of $X$ are of general type.
 Furthermore, if $k=1,2$ and $n\geq \max\{2c+3, c+7\}$, Theorem \ref{thm:Ein+} applies for $a=1$, and we conclude that all the curves and the surfaces in $X$ are of general type.
 Thus $X$ is algebraically hyperbolic (\`a la Lang).

 Now, let $C\subset X$ be an integral curve, with desingularization $\nu\colon\widetilde{C}\longrightarrow C$. 
 Since $n\geq \max\{2c+2, c+5\}$, Theorem \ref{thm:Ein+} applies for $a=k=1$, and hence $h^0 \left(\widetilde{C}, K_{\widetilde{C}}\otimes \nu^*\mathcal O_{{C}}(-1)\right)>0$.
 Then $K_{\widetilde{C}}\otimes \nu^*\mathcal O_{{C}}(-1)$ is effective, and hence $\deg\left(K_{\widetilde{C}}\otimes \nu^*\mathcal O_{{C}}(-1)\right)=2g(C)-2-\mathcal O_{\widetilde{C}}(1)\geq 0$.
 Thus $2g(C)-2\geq \deg \mathcal O_{\widetilde{C}}(1)=(H\cdot C)$, where $H\coloneqq \mathcal O_{{X}}(1)$.
 In particular, $X$ is algebraically hyperbolic (\`a la Demailly). 
\end{proof}
\begin{remark}\label{rmk:kob}
     The Kobayashi hyperbolicity of a very general complete intersection follows from the  Kobayashi hyperbolicity of a very general hypersurface (there are currently no bounds involving the codimension $c$).
     The available bounds (see e.g. \cite{Ber, Br,Cad})  are of course much worse than the expected ones given by algebraic hyperbolicity, because Kobayashi hyperbolicity is a much more difficult problem to study. 
\end{remark}
In order to prove Theorem \ref{thm:RY+}, we introduce two additional results.
The first one is a special case of \cite[Theorem 1.2]{BCFS20}.
\begin{theorem}\label{thm:BCFS} Let $n\geq c+2$ and let $X \subset \mathbb{P}^n$ be a very general complete intersection of type $(d_1, \dots, d_c)$, with $\Pi_{i=1}^c d_i >2$. 
If 
\begin{equation}\label{eq:BCFS}
\sum_{i=1}^c  \binom{d_i+2}{2} \geq 3n,  
\end{equation} 
then the Fano scheme $F_1(X)$ of lines in $X$ contains neither rational nor elliptic curves. 
\end{theorem}

The following result is an extension to complete intersections of \cite[Lemma 2.4]{RY20}, which concerns the case of hypersurfaces.
\begin{lemma}\label{lem:RY+} Let $n\geq c+2$ and let $X \subset \mathbb{P}^n$ be a general complete intersection of type $(d_1, \dots, d_c)$, with 
\begin{equation}\label{eq:RY+}
d\geq \frac{3n-c-1}{2}. 
\end{equation}  
Then $X$ does not contain reducible conics. 
\begin{proof}
    We follow the same argument of the proof of \cite[Lemma 2.4]{RY20}.
    Let $\mathcal{C}$ be the space parameterizing reducible conics in $\mathbb{P}^n$, which has dimension $n+2(n-1)=3n-2$.
    Consider the incidence variety
    $$\mathcal{I}:=\left\{\left.\left([C],\mathbf{F}\right)\in \mathcal{C}\times S^{\mathbf{d}} \right| C\subset X_{\mathbf{F}}\right\},$$
    endowed with the natural projections $\pi_1\colon \mathcal{I}\longrightarrow \mathcal{C}$ and $\pi_2\colon \mathcal{I}\longrightarrow S^{\mathbf{d}}$.
    Let $N\coloneqq\dim S^{\mathbf{d}}$, and notice that vanishing on a fixed reducible conic $C$ imposes exactly $2d_i+1$ conditions to homogeneous polynomials $F_i$ of degree $d_i$. 
    Hence $\dim \pi_1^{-1}([C])=\dim S^{\mathbf{d}}-\sum (2d_i+1)=N-2d-c$, and 
    $$\dim \mathcal{I}=\dim \mathcal{C}+\dim \pi_1^{-1}([C])=3n-2+N-2d-c\leq N-1,$$
    where the inequality follows from \eqref{eq:RY+}. 
    Thus the projection $\pi_2\colon \mathcal{I}\longrightarrow S^{\mathbf{d}}$ is not dominant and the assertion follows.
\end{proof}
\end{lemma}

We can now prove Theorem \ref{thm:RY+}.
\begin{proof}[Proof of Theorem \ref{thm:RY+}]
Let $X\subset \mathbb{P}^n$ be a very general complete intersection of type $(d_1,\dots,d_c)$, with $d\leq 2n-c-2$.
It follows from \cite[Th\'eor\`eme 2.1]{DM} that $X$ contains lines, so we consider the universal family over the Fano scheme $F_1(X)$ of lines in $X$,
$$
\mathcal{L}:=\left\{\left.\left([\ell],x\right)\in F_1(X)\times \mathbb{P}^n \right| x\in \ell\right\}\stackrel{p}{\longrightarrow} F_1(X).
$$
Since $d\geq n+1+\left\lfloor\frac{n-c}{2}\right\rfloor$, Lemma \ref{lem:RY+} yields that $X$ contains no reducible conics.
Hence the projection $q\colon \mathcal{L}\longrightarrow \mathbb{P}^n$ is a bijection over the subvariety $q(\mathcal{L})\subset X$ covered by the lines lying on $X$.

We note that the first quadratic bound in the statement is \eqref{eq:bound quadratico} with $a=0$. 
Therefore, if $Y\subset X$ is a rational curve, then $Y\subset q(\mathcal{L})$ by Theorem \ref{thm:main}, and $q^{-1}(Y)\subset \mathcal{L}$ is a rational curve.
Moreover, using \eqref{eq:bound quadratico}, we obtain
\begin{equation*}
\sum_{i=1}^c  \binom{d_i+2}{2} = \sum_{i=1}^c  \left(\frac{d_i(d_i+1)}{2}+d_i+1\right)\geq 3n-2 + d + c\geq 3n,
\end{equation*}
that is \eqref{eq:BCFS} holds.
Thus Theorem \ref{thm:BCFS} yields that the curve $q^{-1}(Y)\subset \mathcal{L}$ is contracted by the map $p\colon \mathcal{L}\longrightarrow F_1(X)$, i.e. $Y$ is a line.    

As far as the second part of the statement is concerned, suppose that \eqref{eq:bound lineare} and \eqref{eq:bound quadratico} hold with $a=1$.
If $Y\subset X$ were an elliptic curve, by arguing as above, we would deduce that $Y$ is contracted by $q\colon \mathcal{L}\longrightarrow \mathbb{P}^n$, a contradiction. 
Thus $X$ does not contain elliptic curves.
\end{proof}

\begin{remark}\label{rem:(1.2-2)}
We note that if $a=0$, $k=1$ and $c\leq n/3$, then \eqref{eq:bound lineare} implies \eqref{eq:bound quadratico}.
Indeed, \eqref{eq:bound lineare} gives $d\geq \frac{3n+2-c}{2}$, and using \eqref{eq:minorazione}, we have that 
\begin{align*}
        \sum_{i=1}^c \frac{d_i(d_i+1)}{2}& \geq  \frac{1}{2c}\left(d^2+cd\right) \geq \frac{1}{8c}\left[ (3n+2-c)^2 +2c(3n+2-c)\right] \\
        & =\frac{1}{8c}\left[9n^2+12n+4-c^2\right]\geq\frac{3}{8n}\left[9n^2+12n+4-\frac{n^2}{9}\right]
        >3n+4,
    \end{align*}
which implies \eqref{eq:bound quadratico}.  
\end{remark}

\end{document}